\documentclass[smallextended,numbook,runningheads]{svjour3}     
\smartqed  
\usepackage{graphicx}
\usepackage{amsmath}
\usepackage{epstopdf} 
\usepackage{mathptmx}      
\usepackage{amsmath}
\usepackage{amssymb}
\usepackage{amsfonts}
\usepackage{verbatim}
\usepackage{mathrsfs}
\usepackage[active]{srcltx}
\usepackage{dsfont}
\usepackage{lineno} 
\usepackage[latin1]{inputenc}
\usepackage{color}
\usepackage[colorlinks,citecolor=blue,urlcolor=blue]{hyperref}

\newtheorem{tm}{Theorem}[section]
\newtheorem{rk}{Remark}[section]
\newtheorem{ap}{Assumption}[section]
\newtheorem{df}{Definition}[section]
\newtheorem{prop}{Proposition}[section]
\newtheorem{lm}{Lemma}[section]
\newtheorem{cor}{Corollary}[section]
\newtheorem{ex}{Example}[section]
\newtheorem{Con}{Conjecture}[section]

\newcommand{\ee}{\mathbb E}
\newcommand{\ff}{\mathbb F}
\newcommand{\pp}{\mathbb P}
\newcommand{\nn}{\mathbb N}
\newcommand{\rr}{\mathbb R}
\newcommand{\hh}{\mathbb H}
\newcommand{\zz}{\mathbb Z}

\newcommand{\br}{\mathbf r}
\newcommand{\bs}{\mathbf s}

\newcommand{\BB}{\mathcal B}
\newcommand{\CC}{\mathcal C}
\newcommand{\LL}{\mathcal L}
\newcommand{\TT}{\mathcal T}
\newcommand{\OO}{\mathcal O}
\newcommand{\PP}{\mathcal P}

\newcommand{\OOO}{\mathscr O}
\newcommand{\FFF}{\mathscr F}

\newcommand{\<}{\langle}
\renewcommand{\>}{\rangle}
\allowdisplaybreaks \allowdisplaybreaks[4]

\journalname{Stoch PDE: Anal Comp}
\begin{document}


\title{Strong Approximation of Monotone Stochastic Partial Differential Equations Driven by Multiplicative Noise
\thanks{}
}

\titlerunning{Strong Approximation of Monotone SPDEs Driven by Multiplicative Noise}        

\author{
Zhihui Liu \and Zhonghua Qiao 
}


\institute{Zhihui Liu \at
              Department of Mathematics, \\
The Hong Kong University of Science and Technology, \\
Clear Water Bay, Kowloon, Hong Kong \\
              \email{zhliu@ust.hk}
               \\
           \and
           Zhonghua Qiao \at
              Department of Applied Mathematics, \\
The Hong Kong Polytechnic University, \\
Hung Hom, Kowloon, Hong Kong \\
              \email{zhonghua.qiao@polyu.edu.hk}           
\thanks{}
           }

\date{Received: 12 December 2018 / Revised: 22 January 2020 \\ 
https://doi.org/10.1007/s40072-020-00179-2 }

\maketitle

\begin{abstract}
We establish a general theory of optimal strong error estimation for numerical approximations of a second-order parabolic stochastic partial differential equation with monotone drift driven by a multiplicative infinite-dimensional Wiener process.
The equation is spatially discretized by Galerkin methods and temporally discretized by drift-implicit Euler and Milstein schemes. 
By the monotone and Lyapunov assumptions, we use both the variational and semigroup approaches to derive a spatial Sobolev regularity under the $L_\omega^p L_t^\infty \dot H^{1+\gamma}$-norm and a temporal H\"older regularity under the $L_\omega^p L_x^2$-norm for the solution of the proposed equation with an $\dot H^{1+\gamma}$-valued initial datum for $\gamma\in [0,1]$. 
Then we make full use of the monotonicity of the equation and tools from stochastic calculus to derive the sharp strong convergence rates $\OOO(h^{1+\gamma}+\tau^{1/2})$ and 
$\OOO(h^{1+\gamma}+\tau^{(1+\gamma)/2})$ for the Galerkin-based Euler and Milstein schemes, respectively.
\end{abstract}

\maketitle

\linenumbers 


\section{Introduction}
\label{sec1}

There exists a general theory of strong error estimation for numerical approximations of stochastic partial differential equations (SPDEs) with Lipschitz coefficients; see, e.g., \cite{CHL17(SINUM)}, \cite{JR15(FCM)}, \cite{Kru14(SPDE)}, \cite{Kru14(IMA)} and references therein.
For SPDEs with non-Lipschitz coefficients driven by multiplicative noises, there are only a few results on strong approximations.
Since these SPDEs cannot be solved explicitly, one might need to develop efficient numerical techniques to study them. 
Depending on particular physical models, it may be necessary to design numerical schemes for the underlying SPDEs with strong convergence rates, i.e., rates in $L^p(\Omega)$ concerning the sample variables for some $p\ge 1$, see, e.g., \cite{BFRS16(SINUM)} and references quoted therein.

To deal with the nonlinearity of SPDEs, one usually applies the truncation technique, which only produces convergence rate in certain sense such as in probability or pathwise which is weaker than strong sense, see, e.g., \cite{KLL15(JAP)} for the stochastic Allen--Cahn equation and \cite{Liu13(SINUM)} for the stochastic nonlinear Schr\"odinger equation.
A few recent contributions on strong error estimations for numerical approximations of particular types of SPDEs with non-Lipschitz coefficients by using the so-called exponential integrability for both the exact and numerical solutions.
However, such exponential integrability is only shown to hold for several special SPDEs.
In fact, up to our knowledge, only the two-dimensional stochastic Navier--Stokes equation \cite{Dor12(SINUM)}, the one-dimensional stochastic Burgers equation and Cahn--Hilliard--Cook equation \cite{HJ14}, and the one-dimensional stochastic nonlinear Schr\"odinger equation \cite{CHL17(JDE)}, \cite{CHLZ19(JDE)} were proved to possess this type of exponential integrability.

\subsection{Stochastic Allen--Cahn Equation}

Currently, several works are performed to give strong error estimations for semidiscrete or fully discrete schemes of the stochastic Allen--Cahn equation (or stochastic Ginzburg--Landau equation) driven by an additive white or colored noises:
\begin{align} \label{ac}
{\rm d} u=(\Delta u+u-u^3) {\rm d}t
+{\rm d}W(t),
\quad u(0)=u_0.
\end{align}
A major feature of the nonlinear drift term in Eq. \eqref{ac} is that the following monotone (more precisely, one-sided Lipschitz) condition holds: 
\begin{align} \label{ac-mon}
& [(x-x^3)-(y-y^3)] (x-y) \le C (x-y)^2,
\quad x,y\in \rr.
\end{align}
This type of stochastic equation, arising from phase transition in materials science by stochastic perturbation such as impurities of the materials, has been extensively studied in the literature; see, e.g., \cite{BGJK17}, \cite{BJ19(SPA)}, \cite{LQ19(IJNAM)}, \cite{LQ20(IMA)}, \cite{Wan18} for one-dimensional equation with space-time white noise and 
\cite{BCH19(IMA)}, \cite{CH19(SINUM)}, \cite{KLL18(MN)}, \cite{QW19(JSC)} for $d=1,2,3$-dimensional equation with colored noises.
The main ingredient of these approaches in all of the aforementioned papers is the full use of the additive nature of the noise in Eq. \eqref{ac}.

The case of multiplicative noise is more subtle and challenging.
For the $d=1,2,3$-dimensional stochastic Allen--Cahn equation under periodic boundary condition driven by a multiplicative $K\in \nn_+$-dimensional Wiener process 
$\beta=(\beta_1,\cdots, \beta_K)$:
\begin{align} \label{eq-pro}
{\rm d} u=(\Delta u+u-u^3) {\rm d}t
+\sum_{k=1}^K g_k(u) {\rm d}\beta_k(t),
\quad u(0)=u_0,
\end{align}
the author in \cite{Pro14} proved that the drift-implicit Euler--Galerkin finite element scheme \eqref{full0} with $F(u_h^{m+1})=u_h^{m+1}-(u_h^{m+1})^3$ possesses the strong convergence rate
\begin{align*}
& \sup_{0\le m\le M} \|u(t_m)-u_h^m \|_{L_\omega^2 L_x^2} 
=\OOO (h^{\frac13-\epsilon}+\tau^{\frac12-\epsilon} ),
\end{align*}
and sharpened to $h^{1-\epsilon}$ for $d\le 2$ with an infinitesimal factor 
$\epsilon$, under certain smooth and bounded assumptions on $(g_k)$ and the assumption that $u_0\in H^2$.
When $g=g_1 \in \CC_b^2$ (in the case $K=1$), this convergence rate was improved to 
\begin{align} \label{mp}
& \sup_{0\le m\le M} \|u(t_m)-u_h^m \|_{L_\omega^2 L_x^2}  
=\OOO (h+\tau^{\frac12-\epsilon} ),
\end{align}
by \cite{MP18(CMAM)} for a modified scheme of \eqref{full0} where the nonlinear drift term $F(u_h^{m+1})$ was replaced by 
$\frac12({u_h^{m+1}+u_h^m})(1-(u_h^{m+1})^2)$.
We also refer to \cite{GM09(PA)} for a theory of strong approximations of a class of quasilinear, monotone SPDEs driven by finite-dimensional Wiener processes with certain regularity conditions assumed for the solution.

It is an interesting and difficult problem to generalize the above strong error estimations for numerical approximations of Eq. \eqref{eq-pro} driven by an infinite-dimensional Wiener process.
On the other hand, as is well-known that the finite element method possesses an optimal convergence rate for deterministic parabolic PDEs with Lipschitz coefficients, the following conjecture arises naturally.

\begin{Con} \label{Q1}
The drift-implicit Euler--Galerkin scheme \eqref{full0} of the stochastic Allen--Cahn equation driven by a multiplicative infinite-dimensional ${\bf Q}$-Wiener process  (Eq. \eqref{eq-pro} with $K=\infty$) with mild assumptions on the diffusion coefficients satisfies
\begin{align*}
& \sup_{0\le m\le M} \|u(t_m)-u_h^m \|_{L_\omega^2 L_x^2}
=\OO (h^\bs+\tau^\frac12 ),
~~ \text{provided} ~~ u_0\in \dot H^\bs
~~ \text{for}~~ \bs=1, 2.
\end{align*}
\end{Con}

The above conjecture imposes the homogeneous Dirichlet boundary condition on the equation; such question can be readily modified to cover other self-adjoint boundary conditions, such as periodic and homogeneous Neumann boundary conditions.

\subsection{Main Motivations}

The proposed Conjecture \ref{Q1} is one of our main motivations for this study.
In the finite-dimensional case, provided that the drift function 
$f$ satisfies one-sided Lipschitz continuity and polynomial growth conditions and the diffusion function $g$ satisfies global Lipschitz condition (see Remark \ref{rk-hms}), it was shown in \cite{HMS02(SINUM)} that the drift-implicit Euler scheme (i.e., the so-called backward Euler scheme) applied to the stochastic ordinary differential equation (SODE)
\begin{align} \label{sode} 
{\rm d}u(t) &=f(u(t)) {\rm d}t+g(u(t)) {\rm d} \beta(t),
\end{align}
is stable and possesses the standard strong order $1/2$. 
A general result for the infinite-dimensional case under the aforementioned conditions is still unknown and remains an open problem.
This problem forms our second motivation for this study.
 
Our main purpose in the present paper is giving a general theory of optimal strong error estimations for numerical approximations of semilinear SPDEs with a monotone drift which grows at most polynomially such that similarity of the one-sided Lipschitz condition \eqref{ac-mon} holds. 
We focus on the following second-order parabolic SPDE:
\begin{align}\label{see-fg}
{\rm d} u(t,x)  
=(\Delta u(t,x)+f(u(t,x))) {\rm d}t
+g(u(t,x)) {\rm d}W(t,x),
\end{align}
on the physical time-space domain $(t,x)\in \rr_+ \times \OOO$, where 
$\OOO\subset \rr^d$ ($d=1,2,3$) is a bounded domain with smooth  boundary. 
Eq. \eqref{see-fg} is subject to the following initial value and homogeneous Dirichlet boundary condition:
\begin{align}\label{dbc}
u(t,x)=0,\quad (t,x)\in \rr_+\times \partial \OOO; \quad
u(0,x)=u_0(x), \quad x\in \OOO.
\end{align}
Here the drift function $f$ is only assumed to be of monotone-type with polynomial growth which includes the case $f(x)=x-x^3$ for $x\in \rr$, $g$ satisfies the usual Lipschitz condition in infinite-dimensional setting (see Assumptions \ref{ap-f}-\ref{ap-g}), and 
$\{W(t): t\ge 0\}$ is an infinite-dimensional $\bf Q$-Wiener process in a stochastic basis $(\Omega, \FFF, \ff, \pp)$.

To study Eq. \eqref{see-fg}--\eqref{dbc}, we consider an equivalent infinite-dimensional stochastic evolution equation (SEE) 
\begin{align}\label{see} 
{\rm d}u(t) &=(Au(t)+F(u(t))) {\rm d}t+G(u(t)) {\rm d}W(t);
\quad u(0)=u_0,
\end{align}
where $A$ is the Dirichlet Laplacian, $F$ and $G$ are Nemytskiii operators associated to $f$ and $g$, respectively; see Section \ref{sec2.2} for details.
We will use different settings for the solutions of Eq. \eqref{see} and show that they are essentially equivalent (see Lemma \ref{lm-sol}).
Our first main purpose is to derive the optimal strong convergence rate of the drift-implicit Euler--Galerkin finite element scheme (see the scheme \eqref{full} in Section \ref{sec5.1})
\begin{align}\label{full0} 
u_h^{m+1}
=u_h^m+\tau A_h u_h^{m+1}
+\tau \PP_h F(u_h^{m+1})
+\PP_h G(u_h^m) (W(t_{m+1})-W(t_m)),  
\end{align}
with the initial value $u_h^0=\PP_h u_0$ applied to Eq. \eqref{see}.

Besides the optimality of spatial Galerkin approximations, the last motivation is to construct a temporal higher-order scheme for Eq. \eqref{see}.
To construct a higher-order scheme for an SODE \eqref{sode}, a popular and impressively effective numerical scheme is the Milstein scheme based on It\^o--Taylor expansion on the diffusion term. 
Recently, such infinite-dimensional analog of Milstein scheme had been investigated by \cite{JR15(FCM)}, \cite{Kru14(SPDE)} for SPDEs with Lipschitz coefficients which fulfill a certain commutativity type condition. 
Our second main purpose is to derive the sharp strong convergence rate of the Milstein--Galerkin finite element scheme (see the scheme \eqref{milstein} in Section \ref{sec5.2})
\begin{align}\label{milstein0} 
&U_h^{m+1} =U_h^m+\tau A_h U_h^{m+1}
+\tau \PP_h F(U_h^{m+1})
+\PP_h G(U_h^m) (W(t_{m+1})-W(t_m)) \nonumber \\
&\qquad \qquad +\PP_h DG(U_h^m) G(U_h^m) 
\Big[\int_{t_m}^{t_{m+1}} (W(r)-W(t_m)) {\rm d}W(r) \Big], 
\end{align}
with the initial value $u_h^0=\PP_h u_0$ applied to Eq. \eqref{see}.

\subsection{Main Ideas and Results}

The main aim in the present paper is to derive the optimal strong convergence rates of the Galerkin-based Euler or Milstein schemes \eqref{full0} and \eqref{milstein0}, respectively, and thus give a positive answer to Conjecture \ref{Q1} for a more general Eq. \eqref{see-fg}--\eqref{dbc} with merely monotone drift which has polynomial growth and more general initial datum which belongs to $\dot H^{1+\gamma}$ for any $\gamma\in [0,1]$.
To answer Conjecture \ref{Q1}, there are two major difficulties.
The first one is the sharp trajectory Sobolev as well as H\"older regularity for the solution of Eq. \eqref{see}.
Another one is to deal with the non-Lipschitz drift term as well as the multiplicative noise term. 

To overcome the first difficulty, we combine semigroup framework and factorization method with variational framework and monotone condition.
Precisely, we utilize the well-posedness theory for monotone SEEs from \cite{Liu13(JDE)} to conclude the $\dot H^1$-well-posedness and moments' estimation of Eq. \eqref{see} under Assumptions \ref{ap-f}-\ref{ap-g} (see Theorem \ref{tm-h1}).
Such variational method could not handle the $\dot H^2$-wellposedness of Eq. \eqref{see} since the second derivative of $f$ has no upper bound under Assumption \ref{ap-f}; see Remark \ref{rk-h2} for more details.
To get over this obstacle, we use a bootstrapping argument and factorization method to lift the $\dot H^{1+\gamma}$-regularity for any 
$\gamma\in [0,1]$ (see Proposition \ref{prop-h1+} and Theorem \ref{tm-h2}).

To overcome another difficulty for the drift-implicit Euler--Galerkin finite element scheme \eqref{full0}, we introduce an auxiliary process \eqref{aux1} and combine the factorization method with the theory of error estimation for Euler--Galerkin scheme (see Lemma \ref{lm-u-uhm}).
For the Milstein--Galerkin finite element scheme \eqref{milstein0}, we apply the It\^o--Taylor expansion to lift the temporal convergence rate between an analogous auxiliary process \eqref{aux2} and the exact solution of Eq. \eqref{see} (see Lemma \ref{lm-u-Uhm}).
By making full use of the variational method and the one-sided Lipschitz condition \eqref{con-f}, we reduce the convergence rate between the aforementioned two auxiliary processes and the Galerkin-based fully discrete Euler or Milstein schemes to the convergence rate between the auxiliary processes and the exact solution of Eq. \eqref{see}.

Our main results are the following strong convergence rates for both the drift-implicit Euler--Galerkin finite element scheme \eqref{full0} and the Milstein--Galerkin finite element scheme \eqref{milstein0} of Eq. \eqref{see}. 

\begin{tm} \label{main1}
Let $\gamma\in [0,1)$, $u_0\in \dot H^{1+\gamma}$, and Assumptions \ref{ap-f}-\ref{ap-g} hold.
Let $u$ and $u_h^m$ be the solutions of Eq. \eqref{see} and \eqref{full0}, respectively.
Then  
\begin{align}  \label{uhm} 
\sup_{m\in \zz_M} \|u(t_m)-u_h^m \|_{L_\omega^2 L_x^2} 
&=\OO (h^{1+\gamma}+\tau^\frac12 ).
\end{align}
Assume furthermore that $u_0\in \dot H^2$ and Assumption \ref{ap-g2} holds, then \eqref{uhm} is valid with $\gamma=1$.
\end{tm}

\begin{tm} \label{main2}
Let $\gamma\in [0,1)$, $u_0\in \dot H^{1+\gamma}$, and Assumptions \ref{ap-f}-\ref{ap-g} and \ref{ap-f+}-\ref{ap-g+} hold.
Let $u$ and $U_h^m$ be the solutions of Eq. \eqref{see} and \eqref{milstein0}, respectively.
Then 
\begin{align}  \label{Uhm} 
\sup_{m\in \zz_M} \|u(t_m)-U_h^m \|_{L_\omega^2 L_x^2} 
&=\OO (h^{1+\gamma}+\tau^\frac{1+\gamma}2 ).
\end{align}
Assume furthermore that $u_0\in \dot H^2$ and Assumption \ref{ap-g2} holds, then \eqref{Uhm} is valid with $\gamma=1$.
\end{tm}

We note that similar convergence results in Theorems \ref{main1}-\ref{main2} hold true where $h$ is replaced by $N^{-1/d}$ for spectral Galerkin-based fully discrete Euler and Milstein schemes (see Theorems \ref{u-uNm} and \ref{u-UNm}).
The obtained spatial convergence rate in Theorems \ref{main1}-\ref{main2} for finite element method is optimal, since it exactly coincides with the optimal exponent of Sobolev regularity of the solution in Theorem \ref{tm-h1}, Proposition \ref{prop-h1+}, and Theorem \ref{tm-h2}.
The rate for the fully discrete scheme \eqref{full0} is also sharp which removes the infinitesimal factor appearing in \eqref{mp}, agreeing with the standard theory of strong error estimation for the drift-implicit Euler scheme in the globally Lipschitz case; see \cite{Kru14(IMA)} for $\gamma\in [0,1)$.
In fact, in the simplest case of the stochastic heat equation \eqref{see-fg} where $f\equiv 0$ and $g={\rm Id}$, \cite[Theorem 1]{MR07(FCM)} gave a lower error bound $\OOO(\tau^{1/2})$ for temporal numerical approximation once one only uses a total of $[\tau^{-1}]$ evaluations of one-dimensional components for the driving $\bf Q$-Wiener process $W$.

Associated with the aforementioned problems, the temporal approximation result can be considered as a generalization of the finite-dimensional case in \cite{HMS02(SINUM)} to the infinite-dimensional case  without adding any smooth or bounded assumptions on the coefficients of Eq. \eqref{see-fg} (see Remark \ref{rk-hms}).
The spatial approximation result can be seen as an improvement of the convergence rates in \cite{MP18(CMAM)} and \cite{Pro14}; see Remark \ref{rk-pro}.
Theorem \ref{main1} gives more than expected in Conjecture \ref{Q1}; indeed, it shows that the drift-implicit Euler--Galerkin finite element scheme \eqref{full0} applied to the general SPDE \eqref{see} which includes the stochastic Allen--Cahn equation possesses the sharp strong convergence rate under general regularity assumptions on the initial datum.
Finally, the convergence rate for the Milsten--Galerkin finite element scheme \eqref{milstein0} in Theorem \ref{main2} can be seen as a generalization of \cite{JR15(FCM)} and \cite{Kru14(SPDE)} for SPDEs with Lipschitz coefficients to SPDEs with non-Lipschitz but monotone coefficients.

The rest of the paper is organized as follows.
We list some preliminaries including frequently used notations, assumptions, stochastic tools, and formulations for solutions of Eq. \eqref{see} in the next section.
In Section \ref{sec3}, we proposed a detailed well-posedness and regularity analysis for Eq. \eqref{see}.
Rigorous strong error estimations for the aforementioned two fully discrete schemes \eqref{full0} and \eqref{milstein0} are performed in Sections \ref{sec4}-\ref{sec5}, respectively.
The proofs of Theorems \ref{main1}-\ref{main2} and similar strong error estimations for spectral Galerkin-based fully discrete Euler and Milstein schemes are given in these two sections.

\section{Preliminaries}
\label{sec2}

In this section, we first define functional spaces, norms, and notations that will be used throughout the paper.
Then we give the main assumptions and different formulations of solutions for Eq. \eqref{see}.

\subsection{Notations}

Let $T\in \rr^*_+=(0,\infty)$ be a fixed terminal time.
For an integer $M\in \nn_+=\{1,2,\cdots\}$, we denote by $\zz_M:=\{0,1,\cdots,M\}$ and $\zz_M^*=\{1,\cdots,M\}$.
Let $(\Omega,\FFF,\pp)$ be a probability space with a normal filtration 
$\ff:=(\FFF_t)_{t\in [0,T]}$ which forms a stochastic basis.

Throughout, we use the notations $(L_x^q:=L^q(\OOO;\rr),\|\cdot\|_{L_x^q})$ for $q\in [2,\infty]$ and $(H^\bs=H^\bs(\OOO;\rr),\|\cdot\|_{H^\bs})$ for 
$\bs\in \nn_+$ to denote the usual Lebesgue and Sobolev spaces, respectively; when $q=2$, $L_x^2$ is denoted by $H$.
We will use ${\rm Id}$ to denote the identity operator on various finite-dimensional or infinite-dimensional Hilbert spaces such as $\rr^{K \times K}$ for some $K\in \nn_+$ and $H$ if there is no confusion. 
For convenience, sometimes we use the temporal, sample path, and spatial mixed norm $\|\cdot\|_{L_\omega^p L_t^r L_x^q}$ in different orders, such as
\begin{align*}
\|X\|_{L_\omega^p L_t^r L_x^q}
:=\Big(\int_\Omega \Big(\int_0^T \Big(\int_0^1 |X(t,x,\omega)|^q {\rm d}x\Big)^\frac rq {\rm d}t\Big)^\frac pr \pp({\rm d}\omega)\Big)^\frac 1p
\end{align*}
for $X\in L_\omega^p L_t^r L_x^q$,
with the usual modification for $r=\infty$ or $q=\infty$.
For an integer $\bs\in \nn_+$, $(\CC_b^\bs, \|\cdot\|_{\CC_b^\bs})$ is used to denote the Banach space of bounded functions together with its derivatives up to order $\bs$.

Denote by $A: {\rm Dom}(A)\subset H\rightarrow H$ the Dirichlet Laplacian on $H$.
Let $\{(\lambda_k, e_k)\}_{k=1}^\infty$ be an eigensystem of $-A$ with $\{\lambda_k\}_{k=1}^\infty$ being an increasing order. 
Let $N\in \nn_+$ and $V_N:={\rm Span}\{e_1,\cdots,e_N\}$.
Then $A$ is the infinitesimal generator of an analytic $C_0$-semigroup $S(\cdot)=e^{A\cdot}$ on $H$, and thus one can define the fractional powers $(-A)^\theta$ for $\theta\in \rr$ of the self-adjoint and positive operator $-A$.
Let $\dot H^\theta$ be the domain of 
$(-A)^{\theta/2}$ equipped with the norm $\|\cdot\|_\theta$ (related inner product is denoted by $\<\cdot, \cdot\>_\theta$):
\begin{align} \label{h-theta}
\|u\|_\theta:=\|(-A)^\frac\theta2 u\|,
\quad u \in \dot H^\theta.
\end{align}
In particular, one has $\dot H^0=H$, $\dot H^1=V:=H_0^1(\OOO;\rr)$ and 
$\dot H^2=V\cap H^2$.
The inner products (and related norms) of $H$ and $V$ are denoted by 
$\|\cdot \|$ (and $\<\cdot, \cdot\>$) and $\|\cdot \|_1$ (and $\<\cdot, \cdot\>_1$), respectively. 
Let $V^*=\dot H^{-1}$.
Then $V^*$ is the dual space of $V$ (with respect to $\<\cdot, \cdot\>$); the dualization between $V$ and $V^*$ is denoted by ${_{-1}}\<\cdot, \cdot\>_1$. 
We will need the following ultracontractive and smoothing properties of the analytic $\CC_0$-semigroup $S$ for any $t\in (0,T]$, $\mu\ge 0$ and $\rho\in [0,1]$ (see, e.g., \cite[Theorem 6.13 in Chapter 2]{Paz83}):
\begin{align}\label{ana}
\|(-A)^\mu S(t)\|_{\LL(H)} \le C t^{-\mu}, \quad 
\|(-A)^{-\rho} (S(t)-{\rm Id}_H) \|_{\LL(H)} \le Ct^\rho, 
\end{align}
where $(\LL(H; \dot H^\theta), \|\cdot\|_{\LL(H; \dot H^\theta)})$ denotes the space of bounded linear operators from $H$ to $\dot H^\theta$ for $\theta\in \rr$ and $\LL(H):=\LL(H; H)$.
Here and what follows we use $C$, $C_1$, $C_2$, $\cdots$, to denote generic constants independent of various discrete parameters that may be different in each appearance.
For simplicity, we use the notation $a \simeq_p b$ for $a, b \in \rr$ if there exists two constants $C_1=C_1(p)$ and $C_2=C_2(p)$ such that 
$C_1 b \le a \le C_2 b$.

Let $U$ be another separable Hilbert space and ${\bf Q}\in \LL(U)$ be a self-adjoint and nonnegative definite operator on $U$.
Denote by $U_0:={\bf Q}^\frac12 U$ and 
$(\LL_2^\theta:=HS(U_0; \dot H^\theta), \|\cdot\|_{\LL_2^\theta})$ the space of Hilbert--Schmidt operators from $U_0$ to $\dot H^\theta$ for $\theta\in \rr_+$.
The spaces $U$, $U_0$, and $\LL_2^\theta$ are equipped with Borel 
$\sigma$-algebras $\BB(U)$, $\BB(U_0)$, and $\BB(\LL_2^\theta)$, respectively.
Let $W:=\{W(t):\ t\in [0,T]\}$ be a $U$-valued ${\bf Q}$-Wiener process in the stochastic basis $(\Omega,\FFF, \ff,\pp)$, i.e., there exists an orthonormal basis $\{g_k\}_{k=1}^\infty$ of $U$ which forms the eigenvectors of $\bf Q$ subject to the eigenvalues $\{\lambda^{\bf Q}_k\}_{k=1}^\infty$ and a sequence of mutually independent Brownian motions $\{\beta_k\}_{k=1}^\infty $ such that (see \cite[Proposition 2.1.10]{LR15})
\begin{align*}
W(t)
=\sum_{k\in \nn_+} {\bf Q}^\frac12 g_k \beta_k(t)
=\sum_{k\in \nn_+} \sqrt{\lambda^{\bf Q}_k} g_k \beta_k(t),
\quad t\in [0,T].
\end{align*}

\subsection{Main Assumptions}
\label{sec2.2}

Our main conditions on the coefficients of Eq. \eqref{see-fg}--\eqref{dbc} or the equivalent Eq. \eqref{see} are the following two assumptions.

\begin{ap} \label{ap-f}
$f:\rr\rightarrow \rr$ is differentiable and there exist constants $L_f\in \rr$, $L'_f\in \rr_+$ and $q\ge 2$ such that 
\begin{align} 
& (f(x)-f(y)) (x-y) \le L_f (x-y)^2,
\quad x,y\in \rr, \label{con-f} \\
& |f'(x)|\le L'_f(1+|x|^{q-2}),\quad x\in \rr. \label{con-f'}
\end{align}
\end{ap}

Throughout, we assume that 
\begin{align} \label{q}
q\in [2,\infty), \quad d=1,2;\quad 
q\in [2,4], \quad d=3,
\end{align}
and thus the following frequently used Sobolev embedding holds (see Remark \ref{rk-q} for the source of this technique restriction):
\begin{align} \label{emb}
\dot H^1\hookrightarrow L_x^{2(q-1)}\hookrightarrow L_x^q,
\quad d=1,2,3.
\end{align}

From the one-sided Lipschitz condition \eqref{con-f} and the differentiability of $f$, we get  
\begin{align} \label{coe}
f(x) x & =(f(x)-f(0)) (x-0)+f(0) x  
\le C(1+x^2),
\quad x\in \rr,
\end{align}
and 
\begin{align} \label{f'1}
f'(x) \le L_f, \quad x\in \rr.
\end{align}
By mean value theorem and the growth condition \eqref{con-f'}, it is clear that $f$ grows as most polynomially of order $(q-1)$, i.e., 
\begin{align} \label{con-f1}
|f(x)| \le C(1+|x|^{q-1}),
\quad x \in \rr.
\end{align}

\begin{ex} \label{ex-f}
A concrete example such that Assumption \ref{ap-f} holds is the odd polynomials with negative leading coefficients.
More precisely, let $q-1\in \nn_+$ be an odd number and 
\begin{align} \label{ex-f0}
f(x)=-a_{q-1} x^{q-1}+\sum_{k=0}^{q-2} a_k x^k,
\quad a_{q-1}\in \rr_+^*,\ a_k\in \rr,\ k\in \zz_{q-2}, \ x\in \rr,
\end{align}
then one can easily check that the conditions \eqref{con-f}-\eqref{con-f'} hold true.
\end{ex}

Denote by $q'$ the conjugation of $q$ (i.e., $1/q'+1/q=1$), $F: L_x^q \rightarrow L_x^{q'}$ the Nemytskii operator associated with $f$, i.e.,
\begin{align*}
F(u)(x):=f(u(x)),\quad u\in L_x^{q'},\ x\in \OOO,
\end{align*}
and $_{L^{q'}_x} \<\cdot, \cdot\>_{L^q_x}$ the dual between $L^{q'}_x$ and $L^q_x$.
Then it follows from Assumption \ref{ap-f} that  
\begin{align} 
F\ \text{has a continuous} & \text{ extension from}\ L^{q'}_x\ \text{to}\ L^q_x, \label{F-con} \\
_{L^{q'}_x}\<F(u)-F(v), u-v\>_{L^q_x}
\le & L_f \|u-v\|_{L^2_x}^2,
\quad u,v\in L^q_x, \quad \text{and} \label{F-mon} \\
_{L^{q'}_x}\<F(u), u\>_{L^q_x}
\le & C(1+\|u\|_{L^2_x}^2), 
\quad u\in L^q_x. \label{F-coe}
\end{align}

\begin{rk} \label{rk-liu}
Compared with the assumption on $f$ in the case of space-time white noise where instead of \eqref{con-f} the authors in \cite{LQ20(IMA)} and \cite{LQ19(IJNAM)} assumed that 
\begin{align} \label{lq}
(f(x)-f(y)) (x-y) \le L_f |x-y|^2-c |x-y|^q, \quad x,y\in \rr,
\end{align}
for some $L_f\in \rr$ and $c\in \rr^*_+$, Assumption \ref{ap-f} weakens this monotone-type condition due to the smoothness of the noise.
\end{rk}

Denote by $G: H\rightarrow \LL_2^0$ the Nemytskii operator associated with $g$:
\begin{align*}
G(u) v(x):=g(u(x)) v(x),\quad u\in H,\ v\in U_0,\ x\in \OOO.
\end{align*}

\begin{ap} \label{ap-g}
The operator $G:H\rightarrow \LL_2^0$ is Lipschitz continuous, i.e., there exists a constant $L_g\in \rr_+$ such that 
\begin{align}  \label{con-g} 
\|G(u)-G(v)\|_{\LL_2^0} \le L_g \|u-v\|,
\quad u,v\in H.
\end{align}
Moreover, $G(\dot H^1_x)\subset \LL_2^1$ and 
\begin{align} \label{con-g1}
\|G(z) \|_{\LL_2^1} \le L_g(1+ \|z\|_1),
\quad z\in \dot H^1_x.  
\end{align}
\end{ap}

To derive the $\dot H^2$-well-posedness and moments' estimation of Eq. \eqref{see} in Section \ref{sec3.2}, we need the following growth condition of $G$ in $\dot H^{1+\theta}_x$ for some positive $\theta$.

\begin{ap} \label{ap-g2}
There exist constants $\theta\in (0,1)$ and $\rho\ge 1$ such that $G(\dot H^{1+\theta}_x)\subset \LL_2^{1+\theta}$ and 
\begin{align} \label{con-g2}
\|G(z) \|_{\LL_2^{1+\theta}}
\le C(1+ \|z\|^\rho_{1+\theta}),
\quad z \in \dot H^{1+\theta}_x.
\end{align}
\end{ap}

\begin{rk} \label{rk-pro}
Under Majee--Prohl's condition on $g$ of Eq. \eqref{eq-pro} with $K=1$ that $g\in \CC_b^2(\rr; \rr)$  in \cite[Assumption A.1]{MP18(CMAM)}, Assumption \ref{ap-g} holds true.
In fact, in this case we have $U=\rr$ and ${\bf Q}=\rm Id$ and thus 
\begin{align*} 
\|G(u)-G(v)\|_{\LL_2^0}
=\|g(u)-g(v)\| & \le C \|u-v\|,
\quad u,v\in L^2_x,  \\
\|G(z) \|_{\LL_2^1} 
\le \sqrt{\|g(z)\|^2+\|g'(z) \nabla z\|^2} & \le C(1+ \|z\|_1),
\quad z\in \dot H^1_x,  
\end{align*}
for some constant $C$ depending only on $\|g\|_{\CC_b^1}$ but independent of $\|g\|_{\CC_b^2}$.
\end{rk}

\begin{rk} \label{rk-wan}
In the additive noise case $G \equiv {\rm Id}$ the conditions \eqref{con-g}-\eqref{con-g1} are equivalent to the assumption $\|(-A)^{1/2} {\bf Q}^{1/2}\|_{HS(H;H)}<\infty$ which was imposed in \cite[Section 2.1]{QW19(JSC)};
a multiplicative example fulfilling Assumption \ref{ap-g} is given in Example \ref{ex-g}.
\end{rk}

\begin{rk} \label{rk-hms}
Provided that $f:\rr^\br\rightarrow \rr^\br$ and $g:\rr^\br\rightarrow \rr^{\br\times K}$ are $\CC^1$ functions such that the following conditions hold for some constants $L_f \in \rr$,  $L_g, L_f'\in \rr_+$, $l\in \nn$ and for any $x,y\in \rr^\br$ (see \cite[Assumptions 3.1 and 4.1]{HMS02(SINUM)}): 
\begin{align} 
(f(x)-f(y))^\top (x-y) & \le L_f |x-y|^2, \label{con-f+} \\
\|g(x)-g(y)\|_{HS(\rr^K; \rr^\br)} & \le L_g |x-y|,  \label{con-g+}\\
|f(0)|<\infty, \ \
|f(x)-f(y)|^2 & \le L'_f(1+|x|^l+|y|^l) |x-y|^2,   \label{con-f'+}
\end{align}
the authors in \cite[Theorem 5.3]{HMS02(SINUM)} proved that the backward Euler scheme 
\begin{align*}
u^{m+1}=u^m+\tau f(u^{m+1})+g(u^m) (\beta(t_{m+1})-\beta(t_m)),\quad m\in \zz_{M-1};
\quad u^0=u_0,
\end{align*}
for the $\br$-dimensional SODE \eqref{sode} driven by a $K$-dimensional Wiener process $\beta$ is convergent with mean-square order $1/2$.

The first two conditions  \eqref{con-f+} and \eqref{con-g+} are exactly \eqref{con-f} and \eqref{con-g}, respectively, in the finite-dimensional case $U=\rr^K$, ${\bf Q}=\rm Id$, and $H=\rr^\br$.
It is clear that once the last condition \eqref{con-f'+} is valid, our condition \eqref{con-f'} holds true with $q=2+l/2\ge 2$.
\end{rk}

\subsection{Formulations of Solutions}

Before formulating the solutions of Eq. \eqref{see} in variational and semigroup frameworks, we need to introduce several types of Burkholder--Davis--Gundy inequalities, which are frequently used in the regularity analysis and strong error estimations to control the appearing stochastic integrals in both continuous and discrete settings.

We begin with a Burkholder--Davis--Gundy inequality for Hilbert space valued continuous martingales; see, \cite[Theorem 1.1]{MR16(EM)}.

\begin{lm}  \label{lm-bdg} 
Let $(\hh, \|\cdot\|_{\hh})$ be a real separable Hilbert space.
Then for any $p>0$ and any $\hh$-valued continuous $\ff$-martingale $\{N(t):\ t\in [0, T]\}$ with $N(0)=0$, 
\begin{align}  \label{bdg1} 
\ee\Big[\sup_{t\in [0,T]} \|N(t)\|_{\hh}^p \Big] 
\simeq_p \ee\Big[\<N\>_T^\frac p2 \Big].
\end{align}
\end{lm}

If we take $\hh=H$ and $N(\cdot):=\int_0^\cdot \Phi(r){\rm d}W(r)$ with an $\ff$-adapted process $\Phi \in L_\omega^p L_t^2 \LL_2^0$ in Lemma \ref{lm-bdg}, then we get 
$\<N\>_T=\int_0^T \|\Phi(r)\|^2_{\LL_2^0} {\rm d}r$ 
and the following type of Burkholder--Davis--Gundy inequality for stochastic integral:
\begin{align} \label{bdg2}
\ee\Big[\sup_{t\in [0,T]} \Big\|\int_0^t \Phi(r){\rm d}W(r)\Big\|^p\Big]
\simeq_p \ee\Big[\Big(\int_0^T \|\Phi(r)\|^2_{\LL_2^0} {\rm d}r \Big)^\frac p2 \Big].
\end{align}

We also need the following Burkholder--Davis-Gundy inequality for a sequence of $H$-valued discrete martingales, see \cite[Lemma 4.1]{HJ11(FCM)} for the scalar case.

\begin{lm}  
Let $p\ge 2$ and $\{Z_m\}_{m\in \zz_M}$ be a sequence of $H$-valued random variables with bounded $p$-moments such that 
$\ee [Z_{m+1}| Z_0,\cdots,Z_m]=0$, $m\in \zz_{M-1}$.
Then there exists a constant $C=C(p)$ such that 
\begin{align} \label{bdg-dis}
\Big(\ee\Big[\Big\| \sum_{i=0}^m Z_i \Big\|^p \Big]\Big)^\frac1p
\le C \Big( \sum_{i=0}^m \Big(\ee\Big[ \|Z_i\|^p\Big]\Big)^\frac2p  \Big)^\frac12.
\end{align}
\end{lm}

Next, we recall the following definitions of the variational and mild solutions of Eq. \eqref{see}.
The relation of these two types of solutions to general semilinear SPDEs is given in \cite[Appendix G]{LR15}. 
We mainly focus on a solution which is an $\dot H^1$-valued, 
$\ff$-adapted process.

\begin{df} \label{df}
Let $u=\{u(t):\ t\in [0,T]\}$ be an $\dot H^1$-valued, $\ff$-adapted process.
\begin{enumerate}
\item [{(1)}]
$u$ is called a variational solution of Eq. \eqref{see} if  
\begin{align} \label{var}
u(t)=u_0+\int_0^t (A u(r)+F(u(r)) {\rm d}r
+\int_0^t G(u(r)) {\rm d}W(r),
\quad t\in [0,T].
\end{align}
Eq. \eqref{var} takes values in $\dot H^{-1}$ and explained by
\begin{align*}
\<u(t), v\>
=\<u_0, v\>+ \int_0^t {_{-1}}\<A u(r)+F(u(r)), v \>_1 {\rm d}r
+\int_0^t \<v, G(u(r)) {\rm d}W(r)\>,
\end{align*}
for any $t\in [0,T]$ and $v\in \dot H^1$.

\item [{(2)}]
$u$ is called a mild solution of Eq. \eqref{see} if 
\begin{align}\label{mild}
u(t)=S(t)u_0+S*F(u)(t)+S\diamond G(u)(t),
\quad t\in [0,T],
\end{align}
where $S*F(u)(\cdot):=\int_0^\cdot S(\cdot-r) F(u(r)){\rm d}r$ and 
$S\diamond G(u)(\cdot):=\int_0^\cdot S(\cdot-r) G(u(r)){\rm d}W(r)$ denote the deterministic and stochastic convolutions, respectively.
\end{enumerate}
For these two types of formulations, the uniqueness of solutions are both understood in the sense of stochastic equivalence.
\end{df}

All of the appeared equations in this paper are valid in almost surely (a.s.) sense and we omit the a.s. in these equations for simplicity.

The main ingredient to derive the Sobolev and H\"older regularity for the solution $u$ of Eq. \eqref{see} in the next section (see Section \ref{sec3.2} for details) is identifying the variational and mild solutions and combining them with a bootstrap approach.
We have the following equivalence between the variational and mild solutions of Eq. \eqref{see}.

\begin{lm} \label{lm-sol}
Let $p\ge q$, $u_0$ is $\FFF_0$-measurable such that $u_0\in L^p(\Omega; \dot H^1)$, and Assumptions \ref{ap-f}-\ref{ap-g} hold.
Then a variational solution $\{u(t):\ t\in [0,T]\}$ in $L^p(\Omega; \CC([0,T]; \dot H^1))$ of Eq. \eqref{see} is also its mild solution in the sense of Definition \ref{df}.
\end{lm}

\begin{proof}
By \cite[Remark G.0.6]{LR15}, the variational solution $u$ is also a weak solution of Eq. \eqref{see} in the sense that 
\begin{align*}
& \<u(t), v\>
=\<u_0, v\>
+\int_0^t \<u(r), A^* v \> {\rm d}r
+\int_0^t \<F(u(r)), v\> {\rm d}r
+\int_0^t \<G(u(r)) {\rm d}W(r), v\>,
\end{align*}
for any $v\in \dot H^2$ and $t\in [0,T]$.
On the other hand, by the embedding \eqref{emb}, Young inequality, and the Burkholder--Davis--Gundy inequality \eqref{bdg2} we get 
\begin{align*}
& \ee\Big[ \int_0^T \|u(t)\| {\rm d}t
+\int_0^T \|F(u(t))\| {\rm d}t 
+\int_0^T \|G(u(t))\|^2_{\LL_2^0} {\rm d}t \Big] 
\le C \Big(1+\|u\|^q_{L_t^\infty L_\omega^q \dot H^1} \Big),
\end{align*}
which is finite by the assumption $u \in L^p(\Omega; \CC([0,T]; \dot H^1))$ with $p\ge q$.
We conclude by \cite[Proposition G.0.5(i)]{LR15} that this weak solution is also a mild solution.
\end{proof}

\section{Well-posedness and Regularity}
\label{sec3}

Our main aim in this part is to derive the global well-posedness and sharp trajectory Sobolev and H\"older regularity results for Eq. \eqref{see} under Assumptions \ref{ap-f}-\ref{ap-g2}.

\subsection{$\dot H^1$-well-posedness and Moments' Estimation}
\label{sec3.1}

In this part, we will show the well-posedness and moments' estimation of Eq. \eqref{see} in 
$V$, following a known result that Eq. \eqref{see} whose coefficients satisfy certain monotone and Lyapunov conditions has a unique variational solution $\{u(t):\ t\in [0,T]\}$ in the sense of Definition \ref{df} which is essentially bounded a.s. in $V$.

To apply the variational method, let us introduce the Gelfand triple
$\dot H^1=V\hookrightarrow H=L_x^2 \hookrightarrow V^*=\dot H^{-1}$.
By the classical stochastic variational theory, see, e.g., \cite{Liu13(JDE)}, one can prove that the equation
\begin{align} \label{see-LR} 
{\rm d}u(t) &=\mathbf A(u(t)) {\rm d}t+B(u(t)) {\rm d}W(t), 
\quad t\in [0,T];
\quad u(0)=u_0,
\end{align}
whose coefficients satisfy certain monotone and Lyapunov conditions is well-posedness in $V$.
The following result shows that the conditions given in \cite[Theorem 1.1]{Liu13(JDE)} hold for Eq. \eqref{see} under Assumptions \ref{ap-f}-\ref{ap-g}.
We also show that the solution of Eq. \eqref{see} is indeed continuous a.s. and satisfies the moments' estimation \eqref{reg-h1} in $V$.

\begin{tm} \label{tm-h1} 
Let $p\ge q$, $u_0$ is $\FFF_0$-measurable such that $u_0\in L^p(\Omega; \dot H^1)$, and Assumptions \ref{ap-f}-\ref{ap-g} hold.
Then Eq. \eqref{see} exists a unique variational solution $\{u(t):\ t\in [0,T]\}$ in $L^p(\Omega; \CC([0,T]; \dot H^1))$ such that 
\begin{align}\label{reg-h1}
\ee\Big[ \sup_{t\in [0,T]} \|u(t)\|^p_1 \Big] 
+\ee \Big[\int_0^T \|u(t)\|_1^{p-2} \|A u(t) \|^2 {\rm d}t\Big]
 \le C\Big(1+ \ee\Big[ \|u_0\|^p_1\Big] \Big).
\end{align}
\end{tm}

\begin{proof}
We first verify the following conditions for Eq. \eqref{see-LR} with $\mathbf A(u)=A u+F(u)$ and $B(u)=G(u)$ for all $u,v,z\in \dot H^1$ and $u_N\in V_N$:
\begin{align}
\rr \ni \lambda \mapsto {_{-1}}\<\mathbf A(u+\lambda v), z\>_1
\  & \text{is continuous}; \label{LR-con} \\
{_{-1}}\<\mathbf A(u)-\mathbf A(v), u-v\>_1+\|B(u)-B(v)\|^2_{\LL_2^0} 
& \le C \|u-v\|^2;  \label{LR-mon} \\
\<\mathbf A(u_N), u_N \>_1
& \le C(1+\|u_N\|^2_1);  \label{LR-coe} \\
\|\mathbf A(u)\|_{-1} 
&\le C(1+\|u\|^{p-1}_1);  \label{LR-gro} \\
\|B(u)\|_{\LL_2^1} & \le C(1+\|u\|_1). \label{LR-g}
\end{align}
The first statement \eqref{LR-con} follows from \eqref{F-con} and the embedding \eqref{emb};
\eqref{LR-mon} and \eqref{LR-g} follow from the monotonicity \eqref{F-mon} and the Lipschitz continuity \eqref{con-g}.
Indeed, 
\begin{align*} 
{_{-1}}\<(A u+F(u))-(A v+F(v), u-v\>_1
+\|G(u)-G(v)\|^2_{\LL_2^0}
\le (L_f+L_g) \|u-v\|^2.
\end{align*}
Similarly, \eqref{LR-coe} follows from \eqref{F-coe} and \eqref{f'1}:
\begin{align} \label{fn-coe} 
\<A u_N+F(u_N), u_N\>_1
\le C (1+\|u_N\|_1^2)-\|A u_N\|^2.
\end{align} 
To show \eqref{LR-gro} with $p\ge q$, we use the dual argument and the embedding \eqref{emb} to get 
\begin{align} \label{-1}
\|F(u)\|_{-1}
& =\sup_{v\in \dot H^1} \frac{{_{-1}}\<F(u), v\>_1}{\|v\|_1}
\le \sup_{v\in \dot H^1} \|F(u)\|_{L_x^{q'}} \frac{\|v\|_{L_x^q}}{\|v\|_1} \nonumber \\
&\le C \big(1+\|u\|^{q-1}_{L_x^q} \big) 
\le C \big(1+\|u\|^{q-1}_1\big),
\end{align}
and thus $\|A u+F(u)\|_{-1} 
\le \|A u\|_{-1}+\|F(u)\|_{-1} 
\le C(1+\|u\|^{q-1}_1)$ for any $u\in \dot H^1$.
Applying \cite[Theorem 1.1]{Liu13(JDE)} and the assumption that $u_0\in L^p(\Omega;\dot H^1)$, we conclude that Eq. \eqref{see} exists a unique variational solution $u$ which belongs to $L^p(\Omega; L^\infty(0,T; \dot H^1))$ such that $\ee \big[\int_0^T \|u(t)\|_1^{p-2} \|A u(t) \|^2 {\rm d}t \big]<\infty$.

It remains to show that the solution $u$ is indeed continuous in $\dot H^1$ a.s. such that \eqref{reg-h1} holds. 
To lighten the notations, here and after we omit the temporal variable when an integral appears. 
Applying It\^o formula (see, e.g., \cite[Theorem 4.2.5]{LR15}) to the functional $F: H \rightarrow \rr$ defined by $F(v):=\frac1p \|v\|^p$ with $v=\nabla u$ leads to
\begin{align} \label{ito-h1}  
\frac1p \|u(t)\|^p_1 
& =\frac1p \|u_0\|^p_1
+\frac{p-2}2 \int_0^t \|u\|_1^{p-4}  \sum_{k \in \nn_+} \<u, G(u) {\bf Q}^\frac12 g_k\>_1^2 {\rm d}r \nonumber  \\ 
&\quad +\int_0^t \|u\|_1^{p-2} \<u, G(u){\rm d}W(r)\>_1 \nonumber  \\
&\quad +\int_0^t \|u\|_1^{p-2} \Big(\<A u+F(u), u\>_1
+\frac 12 \|G(u)\|^2_{\LL_2^1}\Big) {\rm d}r.
\end{align} 
By Assumptions \ref{ap-f}-\ref{ap-g}, we have 
\begin{align*}
&\<A u+F(u), u\>_1 \le C (1+\|u\|_1^2)-\|A u\|^2, \quad u\in\dot H^2,
\end{align*}
and 
\begin{align*}
& \|u\|_1^{p-4}  \sum_{k \in \nn_+} \<u, G(u) {\bf Q}^\frac12 g_k\>_1^2
+ \|u\|_1^{p-2} \|G(u)\|^2_{\LL_2^1} \le C(1+\|u\|_1^p), \quad u\in\dot H^1,
\end{align*}
from which we obtain
\begin{align} \label{eq-u1}
&\|u(t)\|^p_1+p \int_0^t \|u\|_1^{p-2} \|A u\|^2 {\rm d}r  \nonumber  \\
&\le \|u_0\|^p_1
+\int_0^t C(1+\|u\|_1^p)  {\rm d}r
+p \int_0^t \|u\|_1^{p-2} \<u, G(u) {\rm d}W(r)\>_1.
\end{align}

Since $\int_0^\cdot \|u\|_1^{p-2} \<u, G(u) {\rm d}W(r)\>_1$ is a real-valued continuous martingale, by the Burkholder--Davis--Gundy inequality \eqref{bdg1}, Young and H\"older inequalities, and the condition \eqref{con-g1}, we get an estimation of the stochastic integral in \eqref{eq-u1}:
\begin{align} \label{bdg+}  
& \ee\Big[\sup_{t\in [0,T]} \Big|\int_0^t \|u\|_1^{p-2} \<u, G(u) {\rm d}W(r)\>_1 \Big| \Big]  \nonumber     \\
& \le C~ \ee\Big[\Big(\int_0^T \|u\|_1^{2p-2} 
\|G(u)\|^2_{\LL_2^1} {\rm d}t\Big)^\frac12 \Big]  \nonumber  \\
&\le \frac1{2p} \ee\Big[\sup_{t\in [0,T]} \|u(t)\|_1^p \Big]
+C \int_0^T \Big(1+ \ee\Big[\|u(t)\|^p_1\Big] \Big) {\rm d}t.
\end{align}
Now taking $L^1_\omega L_t^\infty$-norm on both sides of \eqref{eq-u1}, we obtain
\begin{align*}
&\ee\Big[ \sup_{t\in [0,T]}  \|u(t)\|^p_1\Big]
+p \ee \Big[\int_0^T \|u(t)\|_1^{p-2} \|\Delta u(t) \|^2 {\rm d}t\Big]  \\
&\le \ee\Big[\|u_0\|^p_1\Big]
+\int_0^T C \Big(1+\ee\Big[\|u(t)\|_1^p\Big]\Big)  {\rm d}t,
\end{align*}
from which we get \eqref{reg-h1} by Gr\"onwall inequality.
It is clear from Eq. \eqref{ito-h1} that $u$ is indeed continuous in $V$ a.s. by the absolute continuity of Lebesgue and It\^o integrals under the estimation \eqref{reg-h1}.
\end{proof}

\begin{rk} \label{rk-h2}
Unlike \eqref{fn-coe} in Theorem \ref{tm-h1}, the following uniform coercive condition does not hold in the $\dot H^2$-setting: 
\begin{align*} 
\<\PP_N F(u_N), u_N\>_2
\le C(1+\|u_N\|_2^2),  \quad u_N \in \dot H^2,
\end{align*}
since the second derivative of $f$ has no upper bound under Assumption \ref{ap-f}.
Thus one could not use the arguments in Theorem \ref{tm-h1} to get the $\dot H^2$-wellposedness of Eq. \eqref{see}.
We refer to \cite[Theorem 2]{NS20(SPDE)} for an $\dot H^2(\OOO')$-well-posedness result for a similar equation as Eq. \eqref{ac} driven by an affine noise for any compact domain $\OOO' \subset \OOO$ and analogous results in weighted Sobolev spaces on the whole spatial domain $\OOO$ using results from \cite{Kry94(PTRF)} along with several $L^p$-estimates.
\end{rk}

\subsection{Sobolev and H\"older Regularity}
\label{sec3.2}

In the present part, we will derive the Sobolev and H\"older regularity for the solution $u$ of Eq. \eqref{see}.
The first type of regularity is an estimation of $u$ under the $L_\omega^p L_t^\infty \dot H^{1+\gamma}$-norm for general $\gamma\in [0,1]$.
Another type of regularity is a temporal H\"older regularity of $u$ under the $L_\omega^p L^2_x$-norm.
These two types of regularity are both useful, in the rest Sections \ref{sec4}-\ref{sec5}, to derive the optimal strong convergence rate of fully discrete numerical approximations for Eq. \eqref{see}. 

The main ingredients are Lemma \ref{lm-sol}, the identification of the variational and mild solutions to Eq. \eqref{see}, in combination with the factorization method.
We first derive a sharp H\"older regularity of $u$ under the $L_\omega^p L_x^2$-norm.  

\begin{tm} \label{tm-hol}
Let $p\ge q$, $u_0\in L^{p(q-1)} (\Omega; \dot H^1)$, and Assumptions \ref{ap-f}-\ref{ap-g} hold.
There exists a constant $C$ such that 
\begin{align} \label{hol-u}
\sup_{0\le s < t \le T}\frac{\|u(t)-u(s)\|_{L_\omega^p L_x^2}}{(t-s)^{1/2}} 
\le C \Big(1+\|u_0\|^{q-1}_{L_\omega^{p(q-1)} \dot H^1} \Big).
\end{align}
\end{tm}

\begin{proof}
By the mild formulation \eqref{mild} and Minkovskii inequality, we get 
\begin{align*}
\|u(t)-u(s)\|_{L_\omega^p L_x^2} \le I+II+III,
\end{align*}
where
\begin{align*}
I   &=\|(S(t-s)-{\rm Id})u(s)\|_{L_\omega^p L_x^2}, \\
II  &=\Big\|\int_s^t S(t-r) F(u(r)) {\rm d}r \Big\|_{L_\omega^p L_x^2}, \\
III &=\Big\|\int_s^t S(t-r) G(u(r)) {\rm d}W(r) \Big\|_{L_\omega^p L_x^2}.
\end{align*}
The smoothing property \eqref{ana} and the estimation \eqref{reg-h1} yield that 
\begin{align} \label{I}
I & \le \|S(t-s)-{\rm Id}\|_{\LL(H; \dot H^{-1})} 
\|u(s)\|_{L_\omega^p \dot H^1}
\le C(t-s)^\frac12 (1+\|u_0\|_{L_\omega^p \dot H^1}).
\end{align}

For the second term $II$, by Minkovskii inequality, the ultracontractive and smoothing properties \eqref{ana}, the embedding \eqref{emb}, and the estimation \eqref{reg-h1}, we have
\begin{align} \label{II}
II &\le \int_s^t \|S(t-r)\|_{\LL(H)}  
\|F(u(r))\|_{L_\omega^p L_x^2} {\rm d}r \nonumber \\
&\le C \|F(u)\|_{L_t^\infty L_\omega^p L_x^2} (t-s)
\le C \big(1+\|u_0\|_{L_\omega^{p(q-1)} \dot H^1}^{q-1} \big) (t-s).
\end{align}
For the last term $III$, by the Burkholder--Davis--Gundy inequality \eqref{bdg2}, the conditions \eqref{con-g}-\eqref{con-g1}, the smoothing property \eqref{ana}, and the estimation \eqref{reg-h1}, we get 
\begin{align} \label{III}
III &\le \Big(\int_s^t \|S(t-r)\|^2_{\LL(H)} \|G(u(r))\|^2_{L_\omega^p \LL_2^0} {\rm d}r \Big)^\frac12 \nonumber \\
&\le C \|G(u)\|_{L_t^\infty L_\omega^p \LL_2^0}  (t-s)^\frac12 
\le C \big(1+\|u_0\|_{L_\omega^p \dot H^1} \big) (t-s)^\frac12.
\end{align}
Combining the above estimations \eqref{I}-\eqref{III}, we conclude \eqref{hol-u}.
\end{proof}

\begin{rk}
One can use the factorization method as in Proposition \ref{prop-h1+} and then apply Lemma \ref{hol-cha} to derive $u \in L_\omega^p \CC^{\delta'} L_x^2$ for a $\delta'>0$.
However, our main interest here is to derive the sharp H\"older exponent $\delta$ of $u$ such that $u \in \CC^\delta L_\omega^p L_x^2$.
In this case, the derived temporal strong convergence rate of numerical approximations for Eq. \eqref{see} is optimal (and without any infinitesimal factor).
\end{rk}

Next, we study the Sobolev regularity for the solution $u$ of Eq. \eqref{see}.
Note that it seems impossible to use Theorem \ref{tm-h1} by variational approach to derive the $\dot H^2$-well-posedness of Eq. \eqref{see}; see Remark \ref{rk-h2}.
For simplicity, in the rest we always assume that the initial datum $u_0$ is a nonrandom function; a similar argument can handle the case of random initial data which possess finite, sufficiently large algebraic order moments.

The main tool is the following factorization formulas for deterministic and stochastic convolutions:
\begin{align}
S*F(u)(t)
& =\frac{\sin(\pi \alpha)}{\pi}\int_0^t (t-r)^{\alpha-1} S(t-r) F_\alpha(r) {\rm d}r, \label{fa} \\
S\diamond G(u)(t)
&=\frac{\sin(\pi \alpha)}{\pi}\int_0^t (t-r)^{\alpha-1} S(t-r) G_\alpha(r) {\rm d}r, \label{wa}
\end{align}
where $\alpha\in (0,1)$ and
\begin{align*}
F_\alpha(t):&=\int_0^t (t-r)^{-\alpha} S(t-r) F(u(r)){\rm d}r, \\
G_\alpha(t):&=\int_0^t (t-r)^{-\alpha} S(t-r) G(u(r)){\rm d}W(r),
\end{align*}
for $t\in [0,T]$, see, e.g., \cite{HHL19(JDE)}, \cite{HL19(JDE)}, and references cited therein.

To this purpose, we need the following characterization, see, e.g., \cite[Proposition 4.1]{HL19(JDE)}, about the convolution operator
$R_\alpha$ defined by 
\begin{align}\label{ra}
R_\alpha F(t):=\int_0^t (t-r)^{\alpha-1} S(t-r) F(r) {\rm d} r,\quad t\in [0,T].
\end{align}

\begin{lm} \label{hol-cha} 
Let $p>1$, $1/p<\alpha<1$ and $\rho,\theta,\delta\ge 0$.
Then $R_\alpha$ defined by \eqref{ra} 
is a bounded linear operator from $L^p(0,T; \dot H^\rho)$ to $\CC^{\alpha-1/p-(\theta-\rho)/2}([0,T]; \dot H^\theta)$ for $\theta>\rho$ and $\alpha>(\theta-\rho)/2+1/p$.
\end{lm}

Using Lemma \ref{hol-cha} and a bootstrapping argument, we have the following $\dot H^{1+\gamma}$-regularity of the solution to Eq. \eqref{see} for any $\gamma\in (0,1)$.

\begin{prop} \label{prop-h1+}
Let $\gamma\in (0,1)$ and $u_0\in \dot H^{1+\gamma}$.
Under Assumptions \ref{ap-f}-\ref{ap-g}, the solution $u$ of Eq. \eqref{see} belongs to $L^p(\Omega;  \CC([0,T]; \dot H^{1+\gamma}))$ for any $p\ge 1$ such that
\begin{align} \label{reg-h1+}
& \ee\Big[ \sup_{t\in [0,T]} \|u(t)\|^p_{1+\gamma} \Big] 
\le C \Big(1+\|u_0\|^{p(q-1)}_{1+\gamma} \Big).
\end{align}
\end{prop}

\begin{proof}
Let $p>2$ and $\alpha\in (1/p, 1/2)$. 
For any $\gamma\in [0,1)$, Minkovskii and Young inequalities yield that 
\begin{align*}
\|S*F(u)\|_{L_\omega^p L_t^\infty \dot H^{1+\gamma}}
&\le \Big\|\int_0^t \|(-A)^\frac{1+\gamma}2 S(t-r) F(u(r))\| {\rm d}r \Big\|_{L_\omega^p L_t^\infty} \nonumber  \\
&\le C\big\| t^{-\frac{1+\gamma}2} * \|F(u)\| \big\|_{L_\omega^p L_t^\infty} 
\le C T^\frac{1-\gamma}2 \|F(u)\|_{L_\omega^p L_t^\infty L_x^2}.
\end{align*}
Then we get by the growth condition \eqref{con-f'}, the embedding \eqref{emb}, and the estimation \eqref{reg-h1} that 
\begin{align} \label{est-sf}
\|S*F(u)\|_{L_\omega^p L_t^\infty \dot H^{1+\gamma}}
\le C \big(1+\|u\|^{q-1}_{L_\omega^{p(q-1)} L_t^\infty \dot H^1}\big)
\le C \big(1+\|u_0\|^{q-1}_1 \big).
\end{align}

For the stochastic convolution, we use the factorization formula \eqref{wa}.
By the growth condition \eqref{con-g1}, we have
\begin{align*}
\|G_\alpha\|^p_{L_\omega^p L_t^p \dot H^1}
\le C \Big[\int_0^T \Big(\int_0^t (t-r)^{-2\alpha} {\rm d}r \Big)^\frac p2 {\rm d}t\Big] 
\Big(1+\|u\|_{L_\omega^p L_t^\infty\dot H^1}^p \Big)
\le C \Big(1+\|u_0\|^p_1 \Big).
\end{align*}
Thus Lemma \ref{hol-cha} with $\rho=1$ and $\theta=1+\gamma$ for $\gamma \in [0,1-2/p)$ yields that $S\diamond G(u)\in L^p(\Omega; \CC([0,T]; \dot H^{1+\gamma}))$ and 
\begin{align} \label{est-sg}
\|S\diamond G(u)\|_{L_\omega^p L_t^\infty\dot H^{1+\gamma}}
\le C \|G_\alpha\|_{L_\omega^p L_t^p \dot H^1}
\le C \big(1+\|u_0\|_1 \big).
\end{align}
Combining the estimations \eqref{est-sf}-\eqref{est-sg} with the standard estimation for $S(\cdot) u_0$ that
\begin{align} \label{reg-x0}
\|S(\cdot)u_0\|_{L_\omega^p L_t^\infty\dot H^{1+\gamma}} 
\le C\|u_0\|_{1+\gamma},
\end{align}
we conclude \eqref{reg-h1+} for $\gamma \in [0,1-2/p)$.
Taking $p$ large enough, we prove the result for general $\gamma\in [0,1)$. 
\end{proof}

\begin{cor} \label{cor-infty}
Let $u_0\in \dot H^1\cap L_x^\infty$.
Under Assumptions \ref{ap-f}-\ref{ap-g}, the solution $u$ of Eq. \eqref{see} belongs to $L^p(\Omega;  \CC([0,T]; \dot H^1))
\cap L^p(\Omega;  \CC([0,T]; L_x^\infty))$ for any $p\ge 1$ such that
\begin{align} \label{h1+infty}
& \ee\Big[ \sup_{t\in [0,T]} \|u(t)\|^p_1 \Big] 
+\ee\Big[ \sup_{t\in [0,T]} \|u(t)\|^p_{L_x^\infty} \Big] 
\le C \Big(1+\|u_0\|^{p(q-1)}_1+\|u_0\|^p_{L_x^\infty}\Big).
\end{align}
\end{cor}

\begin{proof}
The proof of Proposition \ref{prop-h1+} and the embedding 
$\dot H^{1+\beta}\hookrightarrow L_x^\infty$ for any $\beta\in (1/2,1]$
imply that  
\begin{align*}
& \|S*F(u)\|_{L_\omega^p L_t^\infty L_x^\infty}
+\|S\diamond G(u)\|_{L_\omega^p L_t^\infty L_x^\infty} \\
&\le \|S*F(u)\|_{L_\omega^p L_t^\infty \dot H^{1+\beta}}
+\|S\diamond G(u)\|_{L_\omega^p L_t^\infty \dot H^{1+\beta}} \\
&\le C \big(1+\|u\|^{q-1}_{L_\omega^{p(q-1)} L_t^\infty \dot H^1}\big)
+C \big(1+\|u\|_{L_\omega^p L_t^\infty\dot H^1}\big)
\le C \big(1+\|u_0\|^{q-1}_1 \big).
\end{align*}
Similarly to \eqref{reg-x0}, we have
$\|S(\cdot)u_0\|_{L_\omega^p L_t^\infty L_x^\infty} 
\le C\|u_0\|_{L_x^\infty}$.
Thus we get 
\begin{align} \label{infty0}
\|u\|_{L_\omega^p L_t^\infty L_x^\infty}
& \le \|S(\cdot)u_0\|_{L_\omega^p L_t^\infty L_x^\infty}
+\|S*F(u)\|_{L_\omega^p L_t^\infty L_x^\infty}
+\|S\diamond G(u)\|_{L_\omega^p L_t^\infty L_x^\infty}  \nonumber  \\
&\le C \big(1+\|u_0\|^{q-1}_1+\|u_0\|_{L_x^\infty} \big).
\end{align}
Combining the estimation \eqref{infty0} with the estimation \eqref{reg-h1}, we derive \eqref{h1+infty}.
\end{proof}

Our next step is to show the $\dot H^2$-regularity for the solution of Eq. \eqref{see}.
To this end, we need the following result.

\begin{lm} \label{lm-fx}
Let $\beta\in [0,1/2)$ and \eqref{con-f'} hold.
For any $u\in \dot H^\beta\cap L_x^\infty$, there exists a constant $C$ such that 
\begin{align}   \label{fx1}
\|F(u)\|_\beta
&\le C \big(1+\|u\|^{q-1}_{L_x^\infty} +\|u\|^{q-1}_\beta \big).
\end{align}
\end{lm}

\begin{proof}
We assume that $\beta \in (0,1/2)$, while the inequality \eqref{fx1} for 
$\beta=0$ is trivial.
Recall that for $\beta \in (0,1/2)$ the space $\dot H^\beta$ coincides with the Sobolev--Slobodeckij space $W^{\beta,2}$ whose norm is defined by
\begin{align*}
\|u\|^2_{W^{\beta,2}}
&:= \|u\|^2+\int_{\OOO}\int_{\OOO} 
\frac{|u(x)-u(y)|^2}{|x-y|^{d+2 \beta}} {\rm d}x{\rm d}y,
\quad u \in W^{\beta,2}.
\end{align*}
Then by mean value theorem, the growth condition \eqref{con-f'}, the embedding \eqref{emb}, and Young inequality, we get 
\begin{align*}
\|F(u)\|^2_\beta
&\le C\Big(\|F(u)\|^2+\int_{\OOO}\int_{\OOO} 
\frac{|f(u(x))-f(u(y))|^2}{|x-y|^{d+2 \beta}} {\rm d}x{\rm d}y \Big) \\
&\le C \Big(1+\|u\|^{2(q-1)}_{L_x^{2(q-1)}} \Big) 
+C \Big(1+\|u\|^{2(q-2)}_{L_x^\infty}\Big)
\int_{\OOO}\int_{\OOO} 
\frac{|u(x)-u(y)|^2}{|x-y|^{d+2 \beta}} {\rm d}x{\rm d}y  \\
&\le C \big(1+\|u\|^{2(q-1)}_{L_x^\infty} +\|u\|^{2(q-1)}_\beta \big).
\end{align*}
This completes the proof of \eqref{fx1}.
\end{proof}

\begin{tm} \label{tm-h2}
Let $u_0\in \dot H^2$ and Assumptions \ref{ap-f}-\ref{ap-g2} hold.
Then the solution $u$ of Eq. \eqref{see} is in $L^p(\Omega; \CC([0,T]; \dot H^2))$ for any $p\ge 1$ such that 
\begin{align}\label{reg-h2}
\ee\Big[ \sup_{t\in [0,T]} \|u(t)\|^p_2 \Big]
\le C \Big(1+\|u_0\|^{p \rho(q-1)}_2+\|u_0\|^{p(q-1)^2}_2 \Big).
\end{align}
\end{tm}

\begin{proof}
By Young inequality, the factorization formula \eqref{fa}, and Lemma \ref{lm-fx}, we get for $p>1$, $1/p<\alpha<1$, and $\beta\in (0,1/2)$ that 
\begin{align*}
\|F_\alpha\|_{L_{\omega,t}^p \dot H^\beta}
&\le \big\|t^{-\alpha} * \|F(u)\|_{\dot H^\beta} \big\|_{L_{\omega,t}^p}
\le C T^{1-\alpha}  \|F(u)\|_{L_\omega^p L_t^\infty \dot H^\beta} \\
&\le C T^{1-\alpha} 
 \big(1+\|u\|_{L_\omega^{p(q-1)} L_t^\infty L_x^\infty}^{q-1}
+\|u\|^{q-1}_{L_\omega^{p(q-1)} L_t^\infty \dot H^\beta} \big).
\end{align*}
Since $u_0\in\dot H^2$, Proposition \ref{prop-h1+} yields that
$u \in L^p(\Omega; \CC([0,T]; \dot H^{1+\gamma})$ for any $p\ge 2$ and $\gamma\in [0,1)$ such that \eqref{reg-h1+} holds.
In particular, we take $\gamma\in (1/2,1)$ such that $\dot H^{1+\gamma}\hookrightarrow L_x^\infty$.
Consequently,
\begin{align*}\big
\|F_\alpha\|_{L_{\omega,t}^p \dot H^\beta}
\le C \big(1+\|u\|_{L_\omega^{p(q-1)} L_t^\infty \dot H^{1+\gamma}}^{q-1} \big)
&\le C \big(1+\|u_0\|_2^{(q-1)^2} \big).
\end{align*}
As a result of Lemma \ref{hol-cha} with $\rho=\beta=5/p$ (take $p>10$) and 
$\alpha=1-1/p$, we obtain $S*F(u)\in L^p(\Omega; \CC([0,T]; \dot H^2))$ such that
\begin{align} \label{h2-f}
\|S*F(u)\|_{L_\omega^p L_t^\infty \dot H^2} 
\le C \|F_\alpha\|_{L_{\omega,t}^p \dot H^\beta}
&\le C \big(1+\|u\|^{(q-1)^2}_2 \big).
\end{align}

For the stochastic convolution $S\diamond G(u)$, the condition \eqref{con-g2} and the estimation \eqref{reg-h1+} yield that 
\begin{align*}
\|G_\alpha\|_{L_{\omega,t}^p \dot H^{1+\theta}}
\le C \big(1+\|u\|^\rho_{L_\omega^{p \rho} L_t^\infty \dot H^{1+\theta}} \big) 
\le C \big(1+\|u_0\|^{\rho(q-1)}_{1+\theta} \big)
\le C \big(1+\|u_0\|^{\rho(q-1)}_2 \big),
\end{align*}
with $\alpha\in (1/p,1/2)$.
Taking for example $p>4/\theta>4$ and $\alpha=1/2-1/p$ such that 
$\alpha\in (1/p,1/2)$ and $2\alpha-2/p+1+\theta>2$, we obtain by
Lemma \ref{hol-cha} that
$S\diamond G(u)\in L^p(\Omega; \CC([0,T]; \dot H^2))$ such that
\begin{align*}
\|S\diamond G(u)\|_{L_\omega^p L_t^\infty \dot H^2} 
\le C \|G_\alpha\|_{L_{\omega,t}^p \dot H^{1+\theta}}
\le C \big(1+\|u_0\|^{\rho (q-1)}_2 \big).
\end{align*}
The above inequality, in combination with \eqref{h2-f} and \eqref{reg-x0} with $\gamma=1$, shows \eqref{reg-h2}.
\end{proof}

\begin{rk} 
One could not expect that $u\in \CC([0,T]; L^p(\Omega; \dot H^{2+\epsilon}))$ provided $u_0\in \dot H^{2+\epsilon}$ for some 
$\epsilon\in \rr^*_+$ under Assumption \ref{ap-g}.
This is due to the restriction of the regularity for the stochastic convolution.
For the additive noise case, the equivalent assumption $\|(-A)^{1/2} {\bf Q}^{1/2} \|_{HS(H,H)}<\infty$ in Remark \ref{rk-wan} is not sufficient to ensure that 
\begin{align*} 
\ee\Big[\Big\|\int_0^t S(t-r) {\rm d}W(r)\Big\|_{2+\epsilon}^2\Big]
<\infty, \quad t\in (0,T].
\end{align*}
In fact, one can find a counterexample that  
${\bf Q}=(-A)^{-\alpha}$ for some $\alpha\in \rr$.
Then 
\begin{align*} 
\ee\Big[\Big\|\int_0^t S(t-r) {\rm d}W(r)\Big\|_{2+\epsilon}^2\Big]
=\int_0^t \sum_{k\in \nn_+}\lambda_k^{2+\epsilon-\alpha} e^{-2\lambda_k r}  {\rm d}r
=\sum_{k\in \nn_+}
\frac{\lambda_k^{1+\epsilon-\alpha} (1-e^{-2\lambda_k t})}2,
\end{align*}
which is convergent for $t\in (0,T]$ if and only if 
$\sum_{k\in \nn_+}\lambda_k^{1+\epsilon-\alpha}<\infty$, i.e., $\alpha>1+\epsilon+\frac d2$.
However, the assumption 
$\|(-A)^{1/2} {\bf Q}^{1/2}\|_{HS(H,H)}<\infty$
is equivalent to the convergence of the previous series with $\epsilon=0$.
\end{rk}

Finally, we give another sharp H\"older regularity under $\dot H^\beta$-norm with $\beta\in [0,1]$ for the solution of Eq. \eqref{see} with $\dot H^{1+\gamma}$-valued initial datum for some $\gamma\in [0,1]$.
It is also frequently used in the derivation of the strong convergence rates of numerical approximations in Sections \ref{sec4}-\ref{sec5}.

\begin{cor} \label{cor-hol}
Let $\gamma\in [0,1]$, $u_0\in \dot H^{1+\gamma}$, and Assumptions \ref{ap-f}-\ref{ap-g} hold. 
Assume that Assumption \ref{ap-g2} holds if $\gamma=1$.
Then for any $p\ge 1$ and $\beta\in [0,1]$, there exists a constant $C$ such that 
\begin{align} \label{hol-u1}
\sup_{0\le s < t \le T} \frac{\|u(t)-u(s)\|_{L_\omega^p \dot H^\beta}} 
{|t-s|^{\frac{1+\gamma-\beta}2\wedge \frac12}} 
\le 
\begin{cases}
C (1+\|u_0\|^{q-1}_{1+\gamma});, \quad \gamma\in [0,1); \\
C (1+\|u_0\|^{\rho(q-1)}_2+\|u_0\|^{(q-1)^2}_2), \quad \gamma=1.
\end{cases}
\end{align}
\end{cor}

\begin{proof}
Without loss of generality, let $0\le s<t\le T$.
We use similar idea from Theorem \ref{tm-hol} to show \eqref{hol-u1}: 
\begin{align*}
\|u(t)-u(s)\|_{L_\omega^p \dot H^\beta} \le I_1+II_1+III_1,
\end{align*}
where
\begin{align*}
I_1 &=\|(S(t)-{\rm Id})u(s)\|_{L_\omega^p \dot H^\beta}, \\
II_1  &=\Big\|\int_s^t S(t-r) F(u(r)) {\rm d}r \Big\|_{L_\omega^p \dot H^\beta}, \\
III_1 &=\Big\|\int_s^t S(t-r) G(u(r)) {\rm d}W(r) \Big\|_{L_\omega^p \dot H^\beta}.
\end{align*}
The smoothing property \eqref{ana} and the estimations \eqref{reg-h1+} and \eqref{reg-h2} imply that 
\begin{align} \label{I1}
I_1 
& \le \|S(t-s)-{\rm Id}\|_{\LL(H; \dot H^{-1-\gamma+\beta})} 
\|u(s)\|_{L_\omega^p \dot H^{1+\gamma}}
\le C(t-s)^\frac{1+\gamma-\beta}2 \|u\|_{L_t^\infty L_\omega^p \dot H^{1+\gamma}} \nonumber \\
& \le 
\begin{cases}
C (t-s)^\frac{1+\gamma-\beta}2 (1+\|u_0\|^{q-1}_{1+\gamma});, \quad \gamma\in [0,1); \\
C (t-s)^\frac{1+\gamma-\beta}2 (1+\|u_0\|^{\rho(q-1)}_2+\|u_0\|^{(q-1)^2}_2), \quad \gamma=1.
\end{cases}
\end{align}

For the second term $II_1$, by the ultracontractive property \eqref{ana}, the embedding \eqref{emb}, and the estimation \eqref{reg-h1} we have
\begin{align} \label{II1}
II_1 &\le \|F(u)\|_{L_t^\infty L_\omega^p L_x^2}
\Big[\int_s^t (t-r)^{-\frac\beta2} {\rm d}r \Big]
\le C \big(1+\|u_0\|^{q-1}_1 \big) (t-s)^{1-\frac\beta2}.
\end{align}
For the last term $III_1$, by the Burkholder--Davis--Gundy inequality \eqref{bdg2}, the conditions \eqref{con-g}, \eqref{con-g1}, and \eqref{con-g2}, the smoothing property \eqref{ana}, and the estimations \eqref{reg-h1} and \eqref{reg-h1+}, we get 
\begin{align} \label{III1}
III_1 
\le \|G(u)\|_{L_t^\infty L_\omega^p \LL_2^1}
(t-s)^\frac12
\le C \big(1+\|u_0\|_1 \big) (t-s)^\frac12.
\end{align}
Combining the above estimations \eqref{I1}-\eqref{III1}, we conclude \eqref{hol-u1}.
\end{proof}

\begin{rk} \label{rk-q}
The technical requirement that $q\le 4$ when $d=3$ in \eqref{q} comes from the embedding 
$H^1\hookrightarrow L_x^{2(q-1)}$ to control $\|F(u)\|$ by $\|u\|_1$, which was used in Lemma \ref{lm-sol}, \eqref{II} in Theorem \ref{tm-hol}, \eqref{est-sf} in Proposition \ref{prop-h1+}, and \eqref{II1} in Corollary \ref{cor-hol}. 
\end{rk}

\section{Euler--Galerkin Scheme}
\label{sec4}

Our main aim in this section is to give a Galerkin-based fully discrete scheme and derive its optimal strong convergence rate.
The proof of Theorem \ref{main1} is given at the end of Section \ref{sec4.2}.

\subsection{Drift-implicit Euler--Galerkin Finite Element Scheme}
\label{sec4.1}

Let $h\in (0,1)$, $\TT_h$ be a regular family of partitions of $\OOO$ with maximal length $h$, and $V_h\subset V$ be the space of continuous functions on $\bar \OOO$ which are piecewise linear over $\TT_h$ and vanish on the boundary $\partial \OOO$.
Let $A_h: V_h\rightarrow V_h$ and $\PP_h: V^* \rightarrow V_h$ be the discrete Laplacian and generalized orthogonal projection operator, respectively, defined by 
\begin{align}  
\<A_h u_h, v_h\> & =-\<\nabla u_h, \nabla v_h\>,
\quad u_h, v_h\in V_h, \label{Ah} \\
\<\PP_h u, v_h\> & ={_{-1}}\<u, v_h\>_1,
\quad u\in V^*,\ v_h\in V_h. \label{ph}
\end{align}
Taking $u\in H$ and $v_h=\PP_h u$ in \eqref{ph} and using Cauchy--Schwarz inequality implies the following $H$-contraction property of 
$\PP_h$:
\begin{align}  \label{ph-h}
\|\PP_h u\|\le \|u\|, \quad u\in H.
\end{align}

The finite element approximation for Eq. \eqref{see} is to find an 
$\ff$-adapted $V_h$-valued process $\{u_h(t):\ t\in [0,T]\}$ such that the following variational equality holds for any $t\in [0,T]$ and $v_h\in V_h$:
\begin{align} \label{see-h}
& \<u_h(t), v_h\>+\int_0^t \<\nabla u_h(r), \nabla v_h\> {\rm d}r  \\
&=\<u_h(0), v_h\>+\int_0^t \<F(u_h(r)), v_h\> {\rm d}r
+\int_0^t \<v_h, G(u_h(r)) {\rm d}W(r)\>. \nonumber 
\end{align}
To complement \eqref{see-h}, we set the initial datum to be 
$u_h(0)=\PP_h u_0$.
Then the finite element approximation \eqref{see-h} is equivalent to
\begin{align} \label{fem}
\begin{split}
& {\rm d}u_h(t) =(A_h u_h(t)+\PP_h F(u_h(t))){\rm d}t 
+\PP_h G(u_h(t)) {\rm d}W(t),
\quad t\in [0,T];  \\
& u_h(0)=\PP_h u_0.
\end{split}
\end{align}

Let $M\in \nn_+$ and $\{(t_m,t_{m+1}]:\ m\in \zz_{M-1}\}$ be an equal length subdivision of $(0,T]$ with temporal step-size $\tau=t_{m+1}-t_m$ for each $m\in \zz_{M-1}$.
The drift-implicit Euler (DIE) scheme of the finite element approximation \eqref{fem}, that we call the DIE Galerkin (DIEG) scheme, is to find a $V_h$-valued discrete process $\{u_h^m:\ m\in \zz_M\}$ such that
\begin{align}\label{full} \tag{DIEG}
&u_h^{m+1}
=u_h^m+\tau A_h u_h^{m+1}
+\tau \PP_h F(u_h^{m+1})
+\PP_h G(u_h^m) \delta_m W,
\quad u_h^0=\PP_h u_0,
\end{align}
where $\delta_m W=W(t_{m+1})-W(t_m)$, $m\in \zz_{M-1}$.
This DIEG scheme had been widely studied; see, e.g., \cite{FLZ17(SINUM)}, \cite{Pro14}, and \cite{QW19(JSC)}.

It is clear that the DIEG scheme \eqref{full} is equivalent to the scheme
\begin{align}\label{full+}
u_h^{m+1}=S_{h,\tau} u_h^m+\tau S_{h,\tau} \PP_h F(u_h^{m+1})
+S_{h,\tau} \PP_h G(u_h^m) \delta_m W,
\quad m\in \zz_{M-1},
\end{align}
with initial datum $u_h^0=\PP_h u_0$, where $S_{h,\tau}:=({\rm Id}-\tau A_h)^{-1}$ is a space-time approximation of the continuous semigroup $S$ in one step.
Iterating \eqref{full+} for $m$-times, we obtain 
\begin{align}\label{full-sum}
u_h^{m+1}
=S_{h,\tau}^{m+1} u_h^0+\tau \sum_{i=0}^m S_{h,\tau}^{m+1-i} \PP_h F(u_h^{i+1})
+\sum_{i=0}^m S_{h,\tau}^{m+1-i} \PP_h G(u_h^i) \delta_i W,
\quad m\in \zz_{M-1}.
\end{align}

Throughout we take $\tau\in (0,1)$ when $L_f\le 0$ and $\tau<1/(4 L_f)$ when $L_f\in \rr_+^*$.
Under Assumptions \ref{ap-f}-\ref{ap-g}, it is not difficult to show that the DIEG scheme \eqref{full} is solvable and stable.

\subsection{Strong Convergence Rate of DIEG Scheme}
\label{sec4.2}

In this part, we give the optimal strong convergence rate of the DIEG scheme \eqref{full}.

Denote by $E_{h,\tau}(t)=S(t)-S_{h,\tau}^{m+1} \PP_h$ for $t\in (t_m, t_{m+1}]$ with $m\in \zz_{M-1}$.
The following estimations of $E_{h,\tau}$ play a pivotal role in the error estimation of the DIEG scheme \eqref{full}; see, e.g., \cite[Lemmas 4.3(i) and 4.4(ii)]{Kru14(IMA)}.

\begin{lm} \label{lm-eht}
Let $t\in (0,T]$.
\begin{enumerate}
\item
For $0\le \nu\le \mu\le 2$ and $x\in \dot H^\nu$,
\begin{align} \label{eht1}
\|E_{h,\tau} (t) x\|\le C (h^\mu+\tau^\frac\mu2) t^{-\frac{\mu-\nu}2} \|x\|_\nu.
\end{align}

\item
For $0\le \mu\le 1$ and $x\in \dot H^\mu$,
\begin{align} \label{eht2}
\Big( \int_0^t \|E_{h,\tau}(r) x\|^2 {\rm d}r \Big)^\frac12 
\le C (h^{1+\mu}+\tau^\frac{1+\mu}2) \|x\|_\mu.
\end{align}
\end{enumerate}
\end{lm}

To estimate the solution $u$ of Eq. \eqref{see} and $u_h^m$, we introduce the auxiliary process 
\begin{align}\label{aux1}
\widetilde{u}_h^{m+1}
=S_{h,\tau}^{m+1} u_h^0
+\tau \sum_{i=0}^m S_{h,\tau}^{m+1-i} \PP_h F(u(t_{i+1}))
+\sum_{i=0}^m S_{h,\tau}^{m+1-i} \PP_h G(u(t_i)) \delta_i W,
\end{align}
for $m\in \zz_{M-1}$, where the terms $u_h^{i+1}$ and $u_h^i$ in the discrete deterministic and stochastic convolutions of \eqref{full-sum} are replaced by $u(t_{i+1})$ and $u(t_i)$, respectively.
We start with the following uniform boundedness of $\widetilde{u}_h^m$ in $L_\omega^p L_t^\infty \dot H_x^1$:
\begin{align} \label{aux1-h1}
\sup_{m\in \zz_M} \ee\Big[\|\widetilde{u}_h^m\|^p_1 \Big]
\le C \Big(1+\|u_0\|_1^{p(q-1)} \Big),
\end{align}
provided $u_0\in \dot H_x^1$ and Assumptions \ref{ap-f}-\ref{ap-g} hold, which can be shown in view of the boundedness \eqref{reg-h1} and the well-known uniform boundedness 
\begin{align} \label{shm}
& \sup_{m\in \zz_M} \|S_{h,\tau}^m \PP_h x\| 
\le C \|x\|, \quad x\in H.
\end{align}

Next, we show the strong error estimation between the exact solution $u$ of Eq. \eqref{see} and the auxiliary process $\{\widetilde{u}_h^m\}_{m\in \zz_M}$ defined by \eqref{aux1}.

\begin{lm} \label{lm-u-uhm}
Let $\gamma\in [0,1)$, $u_0\in \dot H^{1+\gamma}$, and Assumptions \ref{ap-f}-\ref{ap-g} hold.
Then for any $p\ge 1$, there exist a constant $C$ such that 
\begin{align} \label{u-uhathm}
\sup_{m\in \zz_M} \|u(t_m)-\widetilde{u}_h^m\|_{L_\omega^p L_x^2} 
\le C (h^{1+\gamma}+\tau^\frac12 ).
\end{align}
Assume furthermore that $u_0\in \dot H^2$ and Assumption \ref{ap-g2} holds, then \eqref{u-uhathm} is valid with $\gamma=1$.
\end{lm}

\begin{proof}
Let $m\in \zz_{M-1}$.
Subtracting the auxiliary process $\widetilde{u}_h^{m+1}$ defined by \eqref{aux1} from the mild formulation \eqref{mild} with $t=t_{m+1}$, i.e., 
\begin{align*} 
u(t_{m+1})
&=S(t_{m+1}) u_0+\int_0^{t_{m+1}} S(t_{m+1}-r) F(u) {\rm d}(r)
+\int_0^{t_{m+1}} S(t_{m+1}-r) G(u) {\rm d}W(r),
\end{align*}
we get
\begin{align} \label{j}
J^{m+1}:=\|u(t_{m+1})-\widetilde{u}_h^{m+1}\|_{L_\omega^p L_x^2}
\le \sum_{i=1}^3 J^{m+1}_i,
\end{align}
where 
\begin{align*} 
J^{m+1}_1 & =\| E_{h,\tau}(t_{m+1}) u_0\|_{L_\omega^p L_x^2}, \\
J^{m+1}_2 & =\Big\|\int_0^{t_{m+1}} S(t_{m+1}-r) F(u) {\rm d}r 
-\tau \sum_{i=0}^m S_{h,\tau}^{m+1-i} \PP_h F(u(t_{i+1})) \Big\|_{L_\omega^p L_x^2},   \\
J^{m+1}_3 & =\Big\|\int_0^{t_{m+1}} S(t_{m+1}-r) G(u) {\rm d}W(r) 
-\sum_{i=0}^m S_{h,\tau}^{m+1-i} \PP_h G(u(t_i)) \delta_i W \Big\|_{L_\omega^p L_x^2}.
\end{align*}
In the sequel, we treat the above three terms one by one.

The estimation \eqref{eht1} with $\mu=\nu=1+\gamma$ yields that 
\begin{align}  \label{j1}
J^{m+1}_1
\le C \big(h^{1+\gamma}+\tau^\frac{1+\gamma}2 \big) \|u_0\|_{1+\gamma},
\quad \gamma\in [0,1].
\end{align}
To deal with the second term, we decompose it into the following two terms:
\begin{align*} 
J^{m+1}_2 
&\le \Big\|\sum_{i=0}^m \int_{t_i}^{t_{i+1}} S(t_{m+1}-r) 
[F(u)-F(u(t_{i+1}))] {\rm d}r \Big\|_{L_\omega^p L_x^2} \\
&\quad +\Big\|\sum_{i=0}^m \int_{t_i}^{t_{i+1}} 
E_{h,\tau}(t_{m+1}-r) F(u(t_{i+1})) {\rm d}r \Big\|_{L_\omega^p L_x^2} 
:=\sum_{i=1}^2 J^{m+1}_{2i}.
\end{align*}
Similarly to \eqref{-1}, we have by the embedding \eqref{emb} that 
\begin{align} \label{f-1}
\|(-A)^{-\frac12} (F(u)-F(v))\| 
& =\sup_{z\in \dot H^1} \frac{{_{-1}}\<F(u)-F(v), z\>_1}{\|z\|_1}  \nonumber \\
& \le C \big(1+\|u\|^{q-2}_{L_x^{2(q-1)}}+\|v\|^{q-2}_{L_x^{2(q-1)}} \big)
\|u-v\| \cdot \sup_{z\in \dot H^1} \frac{\|z\|_{L_x^{2(q-1)}}}{\|z\|_1} \nonumber \\
&\le C \big(1+\|u\|^{q-2}_1+\|v\|^{q-2}_1 \big) \|u-v\|.
\end{align}
By Minkovskii inequality, the uniform boundedness \eqref{ana}, the above dual estimation \eqref{f-1} and the H\"older estimate \eqref{hol-u}, we get
\begin{align*} 
J^{m+1}_{21} 
&\le C \Big(1+\|u\|^{q-2}_{L_t^\infty L_\omega^{2p(q-2)} \dot H^1} \Big)
 \Big[ \sup_{0\le s < t \le T} \frac{\|u(t)-u(s)\|_{L_\omega^{2p} L_x^2}}
{|t-s|^{1/2}} \Big]  \\
&\qquad \times \Big[\sum_{i=0}^m \int_{t_i}^{t_{i+1}} \|S(t_{m+1}-r) \|_{\LL(H; \dot H^1)} (t_{i+1}-r)^\frac12  {\rm d}r \Big] \\
&\le C \big(1+\|u_0\|_1^{(q-1)(q-2)} \big) \tau^\frac12.
\end{align*}
Applying Minkovskii inequality and using \eqref{eht1} with $\mu=1+\gamma$ with $\gamma\in [0,1)$ and $\nu=0$, the embedding \eqref{emb} and the estimations \eqref{reg-h1}, we derive  
\begin{align}  \label{j22}
J^{m+1}_{22}
&\le \sum_{j=0}^m \int_{t_j}^{t_{j+1}} 
\|E_{h,\tau}(\sigma) F(u(t_{m+1-j}))\|_{L_\omega^p L_x^2}   {\rm d}\sigma 
\nonumber \\
&\le C (h^{1+\gamma}+\tau^\frac{1+\gamma}2) 
\|F(u)\|_{L_t^\infty L_\omega^p L_x^2} 
\Big[\sum_{j=0}^m \int_{t_j}^{t_{j+1}} \sigma^{-\frac{1+\gamma}2} {\rm d}\sigma \Big]  \nonumber \\
&\le C \big(h^{1+\gamma}+\tau^\frac{1+\gamma}2\big) 
\big(1+\|u_0\|_1^{q-1} \big),
\quad \gamma\in [0,1).
\end{align}
Combining the above two estimations implies that 
\begin{align}  \label{j2}
J^{m+1}_2\le C \big(h^{1+\gamma}+\tau^\frac12 \big)  
\big(1+ \|u_0\|_1^{q-1}+\|u_0\|_1^{(q-1)(q-2)} \big),
\quad \gamma\in [0,1).
\end{align}

The last term $J^{m+1}_3$ can be decomposed as 
\begin{align*} 
J^{m+1}_3 
&\le \Big\|\sum_{i=0}^m \int_{t_i}^{t_{i+1}} S_{h,\tau}^{m+1-i} \PP_h
[G(u)-G(u(t_i))] {\rm d}W(r) \Big\|_{L_\omega^p L_x^2} \\
&\quad +\Big\|\int_0^{t_{m+1}} E_{h,\tau}(t_{m+1}-r) G(u) {\rm d}W(r) \Big\|_{L_\omega^p L_x^2} 
:=\sum_{i=1}^2 J^{m+1}_{3i}.
\end{align*}
Applying both the discrete and continuous Burkholder--Davis--Gundy inequalities \eqref{bdg-dis} and \eqref{bdg2}, respectively, and using the uniform stability property \eqref{shm}, the conditions \eqref{con-g}-\eqref{con-g1}, and the estimations \eqref{reg-h1}, \eqref{hol-u}, and \eqref{eht1} with $(\mu,\nu)=(1+\gamma,1)$ for $\gamma\in [0,1)$, we get 
\begin{align}  \label{j3} 
J^{m+1}_3
&\le C \Big(\sum_{i=0}^m \int_{t_i}^{t_{i+1}} 
\|S_{h,\tau}^{m+1-i} \PP_h [G(u)-G(u(t_i))]\|^2_{L_\omega^p \LL_2^0} {\rm d}r \Big)^\frac12  \nonumber  \\
&\quad +C \Big(\int_0^{t_{m+1}}
\|E_{h,\tau}(t_{m+1}-r) G(u)\|^2_{L_\omega^p \LL_2^0} {\rm d}r \Big)^\frac12  \nonumber \\
&\le C \Big[ \sup_{0\le s < t \le T}\frac{\|u(t)-u(s)\|_{L_\omega^p L_x^2}}{(t-s)^{1/2}} \Big]
\Big(\sum_{i=0}^m \int_{t_i}^{t_{i+1}} (r-t_i) {\rm d}r \Big)^\frac12 \nonumber  \\ 
&\quad +C \Big(h^{1+\gamma}+\tau^\frac{1+\gamma}2\Big) 
\Big(1+\|u\|_{L_t^\infty L_\omega^p \dot H^1} \Big) 
\Big(\int_0^{t_{m+1}} r^{-\gamma} {\rm d}r \Big)^\frac12 \nonumber  \\
&\le C \big(1+\|u_0\|_1^{q-1} \big) \big(h^{1+\gamma}+\tau^\frac12\big), 
\quad \gamma\in [0,1).
\end{align}
Putting the estimations \eqref{j1}, \eqref{j2}, and \eqref{j3} together results in
\begin{align*} 
J^{m+1}
\le C \big(1+\|u_0\|_1^{q-1}+\|u_0\|_1^{(q-1)(q-2)} \big)
\big(h^{1+\gamma}+\tau^\frac12 \big), 
\quad \gamma\in [0,1).
\end{align*}
This inequality in combination with estimation on $J^0$ that  
\begin{align*} 
J^0=\|u_0-\PP_h u_0\|_{L_\omega^p L_x^2}
\le C h^{1+\gamma} \|u_0\|_{1+\gamma},
\quad \gamma\in [0,1],
\end{align*}
completes the proof of \eqref{u-uhathm} with $\gamma\in [0,1)$.

To show \eqref{u-uhathm} for $\gamma=1$, we just need to give refined estimations for the $J^{m+1}_{22}$ and $J^{m+1}_{32}$ provided $u_0\in \dot H^2$ and \eqref{con-g2} holds.
Applying Minkovskii inequality and using \eqref{eht1} with $\mu=2$ and $\nu=\beta\in (0,1)$, the embedding $\dot H^2 \hookrightarrow \dot H^\beta \cap L_x^\infty$, and the estimations \eqref{reg-h1}, \eqref{h1+infty}, and \eqref{fx1}, we derive  
\begin{align} \label{j22+}
J^{m+1}_{22}
&\le C (h^2+\tau) 
\Big(1+\|u\|^{q-1}_{L_t^\infty L_\omega^p \dot H^\beta}
+\|u\|^{q-1}_{L_t^\infty L_\omega^p L_x^\infty}\Big) 
\Big[\sum_{j=0}^m \int_{t_j}^{t_{j+1}} \sigma^{-\frac{2-\beta}2} {\rm d}\sigma \Big]  \nonumber \\
&\le C \big(h^2+\tau \big) \big(1+\|u_0\|_2^{(q-1)^2} \big).
\end{align}
On the other hand, \eqref{eht1} with $\mu=2$ and $\nu=1+\theta$, the condition \eqref{con-g2} and the estimation \eqref{reg-h1+} imply that 
\begin{align}  \label{j32} 
J^{m+1}_{32}
& \le C (h^2+\tau) \|G(u)\|_{L_t^\infty L_\omega^p \dot H^{1+\theta}} 
\Big(\int_0^T \sigma^{-(1-\theta)} {\rm d}\sigma \Big)^\frac12 \nonumber \\ 
& \le C  \big(1+\|u_0\|^{q-1}_2 \big) \big(h^2+\tau \big).
\end{align}
\end{proof}

\begin{rk}
Assumption \ref{ap-g2} is not needed in the additive noise case $G(u)\equiv G \in \LL_2^1$ for any $u\in H$ when $\gamma=1$.
In fact, the Burkholder--Davis--Gundy inequality \eqref{bdg2} and the estimation \eqref{eht2} with $\mu=1$ show the following sharp estimation:
\begin{align*}
\Big\|\int_0^{t_{m+1}} E_{h,\tau}(t_{m+1}-r) {\rm d}W(r) \Big\|_{L_\omega^p L_x^2}
\le C \Big( \int_0^T \|E_{h,\tau}(r) \|^2_{\LL_2^0} {\rm d}r \Big)^\frac12 
\le C (h^2+\tau) \|G \|_{\LL_2^1}.
\end{align*}
\end{rk}

Combining Lemma \ref{lm-u-uhm} with a variational approach under the monotone Assumption \eqref{ap-f}, we can prove Theorem \ref{main1} on the strong convergence rate between the solution $u$ of Eq. \eqref{see} and the numerical solution $u_h^m$ of the DIEG scheme \eqref{full}, which gives a positive answer to Conjecture \ref{Q1}.

{\bf Proof of Theorem \ref{main1}}
Let $m\in \zz_{M-1}$.
By Minkovskii inequality, we get 
\begin{align*} 
\|u(t_{m+1})-u_h^{m+1}\|_{L_\omega^p L_x^2}
\le \|u(t_{m+1})-\widetilde{u}_h^{m+1}\|_{L_\omega^p L_x^2}
+\|\widetilde{u}_h^{m+1}-u_h^{m+1}\|_{L_\omega^p L_x^2}.
\end{align*}
In terms of \eqref{u-uhathm}, it suffices to prove that 
\begin{align}  \label{u-uhm2}
\sup_{m\in \zz_M} \|\widetilde{u}_h^m-u_h^m\|_{L_\omega^2 L_x^2}
\le C \big(h^{1+\gamma}+\tau^\frac12 \big),
\quad \gamma\in [0,1].
\end{align}

Define $\widetilde{e}_h^{m+1}:=\widetilde{u}_h^{m+1}-u_h^{m+1}$.
Then $\widetilde{e}_h^{m+1}\in V_h$
with vanishing initial datum $\widetilde{e}_h^0=0$.
In terms of Eq. \eqref{full+} and \eqref{aux1}, it is not difficult to show that 
\begin{align*}
u_h^{m+1}
&=u_h^m+\tau A_h u_h^{m+1}
+\tau \PP_h F(u_h^{m+1})
+\PP_h G(u_h^m) \delta_m W, \\
\widetilde{u}_h^{m+1}
&=\widetilde{u}_h^m+\tau A_h \widetilde{u}_h^{m+1}
+\tau \PP_h F(u(t_{m+1}))
+\PP_h G(u(t_m)) \delta_m W.
\end{align*}
Consequently, 
\begin{align*}
\widetilde{e}_h^{m+1}-\widetilde{e}_h^m
&=\tau A_h \widetilde{e}_h^{m+1}
+\tau \PP_h (F(u(t_{m+1}))-F(u_h^{m+1})) \nonumber \\
&\quad +\PP_h (G(u(t_m))-G(u_h^m)) \delta_m W.
\end{align*}
Multiplying $\widetilde{e}_h^{m+1}$ on both sides of the above equation, we obtain
\begin{align} \label{eq-u-uhm}
& \<\widetilde{e}_h^{m+1}-\widetilde{e}_h^m, \widetilde{e}_h^{m+1}\>
+\tau \|\nabla \widetilde{e}_h^{m+1}\|^2 \nonumber \\
&=\tau ~{_{-1}}\<F(u(t_{m+1}))-F(\widetilde{u}_h^{m+1}), \widetilde{e}_h^{m+1}\>_1
+\tau ~ {_{-1}}\<F(\widetilde{u}_h^{m+1})-F(u_h^{m+1}), \widetilde{e}_h^{m+1}\>_1 \nonumber \\
&\quad +\int_{t_m}^{t_{m+1}} \<\widetilde{e}_h^{m+1}, (G(u(t_m))-G(u_h^m)) {\rm d}W(r)\>.
\end{align}
By the dual estimation \eqref{f-1}, the one-sided Lipschitz condition \eqref{F-mon}, Cauchy--Schwarz inequality and the embedding \eqref{emb}, the right-hand side term of the above equation can be controlled by
\begin{align}\label{k0-}
& \tau \|(-A)^{-\frac12} (F(u(t_{m+1}))-F(\widetilde{u}_h^{m+1})\|
\times \|\nabla \widetilde{e}_h^{m+1}\|
+L_f \tau \|\widetilde{e}_h^{m+1}\|^2 \nonumber \\
&\quad +\int_{t_m}^{t_{m+1}} \<\widetilde{e}_h^{m+1}, (G(u(t_m))-G(u_h^m)) {\rm d}W(r)\>  \nonumber \\
&\le \frac \tau2 \|\nabla \widetilde{e}_h^{m+1}\|^2 
+L_f \tau \|\widetilde{e}_h^{m+1}\|^2 \nonumber \\
&\quad +C \tau \|u(t_{m+1})-\widetilde{u}_h^{m+1}\|^2 
\big(1+\|u(t_{m+1})\|^{q-2}_1 +\|\widetilde{u}_h^{m+1}\|^{q-2}_1 \big) 
\nonumber \\
&\quad +\int_{t_m}^{t_{m+1}} \<\widetilde{e}_h^{m+1}, (G(u(t_m))-G(u_h^m)) {\rm d}W(r)\>.
\end{align}

The elementary identity $(a-b)a=\frac12 (a^2-b^2)+\frac12 (a-b)^2$ for $a,b\in \rr$ yields that 
\begin{align}\label{k0}
\<\widetilde{e}_h^{m+1}-\widetilde{e}_h^m, \widetilde{e}_h^{m+1} \>
&= \frac12 \big(\|\widetilde{e}_h^{m+1}\|^2-\|\widetilde{e}_h^m\|^2 \big)
+\frac12 \|\widetilde{e}_h^{m+1}-\widetilde{e}_h^m\|^2.
\end{align}
To deal with the stochastic term in \eqref{eq-u-uhm}, we use the martingale property of the stochastic integral, Cauchy--Schwarz inequality, and It\^o isometry to deduce that 
\begin{align} \label{k1-}
& \ee\Big[\int_{t_m}^{t_{m+1}} \<\widetilde{e}_h^{m+1}, (G(u(t_m))-G(u_h^m)) {\rm d}W(r)\>\Big] \nonumber \\
&= \ee\Big[\int_{t_m}^{t_{m+1}} \<\widetilde{e}_h^{m+1}-\widetilde{e}_h^m,
(G(u(t_m))-G(u_h^m)) {\rm d}W(r)\>\Big] \nonumber \\
&\le \frac12  \ee\Big[\|\widetilde{e}_h^{m+1}-\widetilde{e}_h^m\|^2\Big]+L_g^2 \ee\Big[\|u(t_m)-\widetilde{u}_h^m\|^2 \Big] \tau
+L_g^2 \ee\Big[\|\widetilde{e}_h^m\|^2 \Big] \tau.
\end{align}
Now taking expectation in \eqref{eq-u-uhm}, we use \eqref{k0-}-\eqref{k1-} to get 
\begin{align*} 
& \frac12 \Big(\ee\Big[\|\widetilde{e}_h^{m+1}\|^2\Big]
-\ee\Big[\|\widetilde{e}_h^m\|^2\Big]\Big) \nonumber \\
&\le L_f \ee\Big[\|\widetilde{e}_h^{m+1}\|^2\Big] \tau
+C \Big(\ee\Big[\|u(t_{m+1})-\widetilde{u}_h^{m+1}\|^4\Big]\Big)^\frac12 \tau \nonumber \\
&\quad \times \Big(1+\Big(\ee\Big[\|u(t_{m+1})\|_1^{2(q-2)}\Big] \Big)^\frac12
+\ee\Big[\|\widetilde{u}_h^{m+1}\|_1^{2(q-2)}\Big] \Big)^\frac12 \Big)  \nonumber \\
&\quad +L_g^2 \ee\Big[\|u(t_m)-\widetilde{u}_h^m\|^2 \Big] \tau
+L_g^2 \ee\Big[\|\widetilde{e}_h^m\|^2 \Big] \tau.
\end{align*}
Then by the estimations \eqref{reg-h1}, \eqref{aux1-h1}, and \eqref{u-uhathm}, we obtain
\begin{align} \label{eq-u-uhm1}
\big(1-2L_f \tau \big) \ee\Big[\|\widetilde{e}_h^{m+1}\|^2\Big]
\le C \tau \big(h^{1+\gamma}+\tau^\frac12\big)^2
+\big(1+2L_g^2 \tau \big) \ee\Big[\|\widetilde{e}_h^m\|^2\Big].
\end{align}
Summing over $m=0,1,\cdots,l-1$ with $1\le l\le M$, we obtain
\begin{align*}
\big(1-2 L_f \tau\big) \ee\Big[\|\widetilde{e}_h^l\|^2\Big]
\le C \big(h^{1+\gamma}+\tau^\frac12\big)^2
+ 2 \big(L_f+L_g^2\big) \tau \sum_{m=0}^{l-1} \ee\Big[\|\widetilde{e}_h^m\|^2\Big],
\end{align*}
from which we conclude \eqref{u-uhm2} by the classical discrete Gr\"onwall inequality and thus completes the proof.

\begin{rk}
Our analysis can be extended to a slight general nonlinear drift $\tilde F =\tilde F(u; \nabla u)=F(u)+ C \nabla u$, where $F$ is the Nemytskii operator associated with $f$ that satisfies Assumption \eqref{ap-f} and $C$ is a constant.
Indeed, for the estimate of $u(t_m)-\widetilde{u}_h^m$ in Lemma \ref{lm-u-uhm}, the corresponding term $J_2^{m+1}$ involving $F(u)$ has the same estimate as \eqref{j2};
for the estimate of $\widetilde{e}_h^{m+1}$ in the proof of Theorem \ref{main1}, one only needs to estimate the term $C \tau \<\nabla (u(t_{m+1})-\widetilde{u}_h^{m+1}), \widetilde{e}_h^{m+1}\>+ C \tau \<\nabla \widetilde{e}_h^{m+1}, \widetilde{e}_h^{m+1}\>$ added in the right-hand side of \eqref{eq-u-uhm}.
It can be bounded by $C \tau \|u(t_{m+1})-\widetilde{u}_h^{m+1}\|^2 
+ C \tau \|\widetilde{e}_h^{m+1}\|^2 
+ \frac \tau2 \|\nabla \widetilde{e}_h^{m+1}\|^2$, that appeared also in \eqref{k0-}, so the proof of Theorem \ref{main1} suffices.
\end{rk}

\subsection{Drift-implicit Euler--Spectral Galerkin Scheme}
\label{sec4.3}

The parallel result, Theorem \ref{main1}, for DIEG scheme \eqref{full} in Section \ref{sec4.1} is also valid for the spectral Galerkin-based Euler scheme of Eq. \eqref{see}.

Let us first recall and introduce more related notations in Section \ref{sec3.1}.
Given an $N\in \nn_+$, $V_N$ is the linear space spanned by the first $N$ eigenvectors of Dirichlet Laplacian.
Let $\PP_N: H \rightarrow V_N$ be the orthogonal projection operator in $H$ and $A_N$ be the Laplacian restricted in $V_N$.

The spectral Galerkin approximation for Eq. \eqref{see} is to find a sequence of $\ff$-adapted $V_N$-valued processes $\{u_N(t):\ t\in [0,T]\}$ such that 
\begin{align}\label{spe}
du_N(t)
&=(A_N u_N(t)+\PP_N F(u_N(t))){\rm d}t +\PP_N G(u_N(t)) {\rm d}W(t),
\quad t\in [0,T],
\end{align}
with initial datum $u_N(0)=\PP_N u_0$.
By the arguments in the proof of Theorem \ref{main1}, one can show the similar solvability and convergence results for the DIE spectral Galerkin (DIESG) scheme
\begin{align}\label{full-n} \tag{DIESG}
&u_N^{m+1}
=u_N^m+\tau A_N u_N^{m+1}
+\tau \PP_N F(u_N^{m+1})
+\PP_N G(u_N^m) \delta_m W,
\quad u_N^0=\PP_N u_0.
\end{align}

\begin{tm}  \label{u-uNm}
Let $\gamma\in [0,1)$, $u_0\in \dot H^{1+\gamma}$, and Assumptions \ref{ap-f}-\ref{ap-g} hold.
For any $p\ge 1$, there exist a constant $C$ such that 
\begin{align}
\sup_{m\in \zz_M} \|u(t_m)-u_N^m\|_{L_\omega^2 L_x^2} 
\le C (N^{-\frac{1+\gamma} d}+\tau^\frac12 ). 
 \label{u-unm}
\end{align}
Assume furthermore that $u_0\in \dot H^2$ and Assumption \ref{ap-g2} holds, then \eqref{u-unm} is valid with $\gamma=1$.
\end{tm}

\section{Milstein--Galerkin Scheme}
\label{sec5}

Our main aim in this last section is to construct a temporal higher-order Galerkin-based fully discrete Milstein scheme for Eq. \eqref{see} based on It\^o--Taylor expansion on the diffusion term.
We will give the proof of Theorem \ref{main2} at the end of Section \ref{sec5.2}.

The Milstein scheme of the finite element approximation \eqref{fem} is to find a $V_h$-valued discrete process $\{U_h^m:\ m\in \zz_M\}$ such that 
\begin{align}\label{milstein} \tag{DIEMG}
\begin{split}
U_h^{m+1}
&=U_h^m+\tau A_h U_h^{m+1}
+\tau \PP_h F(U_h^{m+1})
+\PP_h G(U_h^m) \delta_m W \nonumber \\
&\quad +\PP_h DG(U_h^m) G(U_h^m) 
\Big[\int_{t_m}^{t_{m+1}} (W(r)-W(t_m)) {\rm d}W(r) \Big], 
\end{split}
\end{align}
with initial datum $U_h^0=\PP_h u_0$.
This scheme is derived by adding the last term in \eqref{milstein} to the DIEG scheme \eqref{full} and followed from an It\^o formula to $G(u)$.
We call \eqref{milstein} the DIE Milstein Galerkin (DIEMG) scheme.
This type of DIEMG scheme had been studied and reviewed in e.g. \cite{JR15(FCM)}, \cite{Kru14(SPDE)}, and references therein.

It is clear that the DIEMG scheme \eqref{full} is equivalent to the following compact scheme for $m\in \zz_{M-1}$:
\begin{align}\label{milstein+}
U_h^{m+1}
& =S_{h,\tau} U_h^m+\tau S_{h,\tau} \PP_h F(U_h^{m+1})
+S_{h,\tau} \PP_h G(U_h^m) \delta_m W \nonumber \\
&\quad +S_{h,\tau} \PP_h DG(U_h^m) G(U_h^m) 
\Big[\int_{t_m}^{t_{m+1}} (W(r)-W(t_m)) {\rm d}W(r) \Big],
\end{align}
with initial datum $U_h^0=\PP_h u_0$, where $S_{h,\tau}$ is given in the previous section.
Iterating \eqref{milstein+} for $m$-times, we obtain 
\begin{align}\label{milstein-sum}
U_h^{m+1}
&=S_{h,\tau}^{m+1} u_h^0+\tau \sum_{i=0}^m S_{h,\tau}^{m+1-i} \PP_h F(U_h^{i+1})
+\sum_{i=0}^m S_{h,\tau}^{m+1-i} \PP_h G(U_h^i) \delta_i W 
\nonumber \\
&\quad +\sum_{i=0}^m S_{h,\tau}^{m+1-i} 
\PP_h DG(U_h^i) G(U_h^i) 
\Big[\int_{t_i}^{t_{i+1}} (W(r)-W(t_i)) {\rm d}W(r) \Big].
\end{align}

It should be noted that in the case of additive noise, the last term of the DIEMG scheme \eqref{milstein} vanishes and thus it reduces to the DIEG scheme \eqref{full} (with $G(u)\equiv G$ for any $u\in H$).
In the following, we first show the DIEMG scheme \eqref{milstein} or equivalently, DIEG scheme \eqref{full} applied to Eq. \eqref{see} with additive noise possesses a temporal higher-order of convergence.
Then we pass to the general multiplicative noise which satisfies the additional Assumption \ref{ap-g+}.

\subsection{Additive Noise}
\label{sec5.1}

In the additive noise case, for simplicity, $G$ is a time independent constant operator, the conditions \eqref{con-g}-\eqref{con-g1} are equivalent to the assumption $G\in \LL_2^1$ which was imposed in \cite{QW19(JSC)} as $\|(-A)^{1/2} {\bf Q}^{1/2} \|_{HS(H;H)}<\infty$ with $G\equiv {\rm Id}$.
We can show more temporal convergence order for the fully discrete scheme \eqref{full}, provided that the following growth condition for the second-order derivative of $f$ holds.

\begin{ap} \label{ap-f+}
$f:\rr\rightarrow \rr$ is second time differentiable and there exist constants $L''_f\in \rr^*_+$ and $\widetilde{q}\ge 3$ with $\widetilde{q}\in [3,\infty)$ for $d=1,2$ and $\widetilde{q}\in [3,5]$ for $d=3$ such that 
\begin{align}  \label{con-f''}
|f''(x)|\le L''_f(1+|x|^{\widetilde{q}-3}),\quad x\in \rr.
\end{align}
\end{ap}

The condition \eqref{con-f''} is consistent with \eqref{con-f'}.
In fact, for an odd $(q-1)$-order polynomial \eqref{ex-f0} with $q$ satisfying \eqref{q}, the conditions \eqref{con-f'} and \eqref{con-f''} are valid with 
$\widetilde{q}=q$ when $q\ge 3$ and $\widetilde{q}=3$ when $q<3$ (i.e., $q=2$).

\begin{cor}  \label{cor-u-uhm}
Let $\gamma\in [0,1]$, $u_0\in \dot H^{1+\gamma}$, $G\in \LL_2^1$, and 
Assumptions \ref{ap-f} and \ref{ap-f+} hold.
Let $u$ and $u_h^m$ be the solutions of Eq. \eqref{see} and \eqref{full} with $G(u)\equiv G$ for any $u\in H$, respectively.
Then there exists a constant $C$ such that
\begin{align} \label{u-uhm0}
\sup_{m\in \zz_M} \|u(t_m)-u_h^m\|_{L_\omega^2 L_x^2}
\le C (h^{1+\gamma}+\tau^\frac{1+\gamma}2 ).
\end{align}
\end{cor}

\begin{proof}
Let $\gamma\in [0,1)$ and $\gamma=1$ when \eqref{con-g2} holds.
According to the proof in Theorem \ref{main1}, it suffices to show the following refined estimation for any $p \ge 2$ between the exact solution $u$ and the auxliary process $\widetilde{u}_h^m$ defined by \eqref{aux1}:
\begin{align} \label{u-uhm1+}
\sup_{m\in \zz_M} \|u(t_m)-\widetilde{u}_h^m\|_{L_\omega^p L_x^2}
\le C (h^{1+\gamma}+\tau^\frac{1+\gamma}2 ).
\end{align}
By means of the proof in Lemma \ref{lm-u-uhm}, we only need to give the following two refined estimations for $J^{m+1}_{21}$ and $J^{m+1}_3$ under the additional condition \eqref{con-f''}:
\begin{align} 
J^{m+1}_{21} 
&=\Big\|\sum_{i=0}^m \int_{t_i}^{t_{i+1}} S(t_{m+1}-r) 
[F(u(r))-F(u(t_{i+1}))] {\rm d}r \Big\|_{L_\omega^p L_x^2}
\le C \tau^\frac{1+\gamma}2, \label{j21} \\
J^{m+1}_3 
&=\Big\|\int_0^{t_{m+1}} 
E_{h,\tau}(t_{m+1}-r) G {\rm d}W(r) \Big\|_{L_\omega^p L_x^2}
\le C \Big(h^{1+\gamma}+\tau^\frac{1+\gamma}2 \Big).  \label{j3+} 
\end{align}

The last estimation \eqref{j3+} can be handled by the Burkholder--Davis--Gundy inequality \eqref{bdg2} and \eqref{eht2} with $\mu=\gamma$:
\begin{align*} 
J^{m+1}_3  
\le C \Big(\int_0^{t_{m+1}}  \|E_{h,\tau}(\sigma) G \|^2_{\LL_2^0} {\rm d}\sigma \Big)^\frac12  
\le C \|G \|_{\LL_2^1} \Big(h^{1+\gamma}+\tau^\frac{1+\gamma}2 \Big).
\end{align*}
To show the first estimation \eqref{j21}, note that 
\begin{align*} 
u(t_{i+1})
&=S(t_{i+1}-r) u(r)
+\int_r^{t_{i+1}} S(t_{i+1}-\sigma) F(u(\sigma)) {\rm d}\sigma  \\
&\quad +\int_r^{t_{i+1}} S(t_{i+1}-\sigma) G {\rm d}W(\sigma),
\quad r\in [t_i,t_{i+1}).
\end{align*}
Then using the Taylor formula leads to the splitting of $J^{m+1}_{21}$ into 
\begin{align*} 
J^{m+1}_{21} 
&\le \Big\|\sum_{i=0}^m \int_{t_i}^{t_{i+1}} S(t_{m+1}-r) 
DF(u(r)) (S(t_{i+1}-r)-{\rm Id}) u(r) {\rm d}r \Big\|_{L_\omega^p L_x^2} \\
&\quad +\Big\|\sum_{i=0}^m \int_{t_i}^{t_{i+1}} S(t_{m+1}-r) 
DF(u(r)) \Big[\int_r^{t_{i+1}} S(t_{i+1}-\sigma) F(u(\sigma)) {\rm d}\sigma\Big] {\rm d}r \Big\|_{L_\omega^p L_x^2} \\
&\quad +\Big\|\sum_{i=0}^m \int_{t_i}^{t_{i+1}} S(t_{m+1}-r) 
DF(u(r)) \Big[\int_r^{t_{i+1}} S(t_{i+1}-\sigma) G {\rm d}W(\sigma) \Big]  {\rm d}r \Big\|_{L_\omega^p L_x^2} \\
&\quad +\Big\|\sum_{i=0}^m \int_{t_i}^{t_{i+1}} S(t_{m+1}-r) 
R_F(u(r), u(t_{i+1})) {\rm d}r \Big\|_{L_\omega^p L_x^2}
=:\sum_{i=1}^4 J^{m+1}_{21i}, 
\end{align*}
where $R_F$ denotes the remainder term
\begin{align*} 
& R_F(u(r), u(t_{i+1})) \\
&:=\int_0^1 D^2 F \big(u(r)+\lambda (u(t_{i+1})-u(r)) \big)
 \big(u(t_{i+1})-u(r), u(t_{i+1})-u(r)\big) (1-\lambda) {\rm d}\lambda \\
&=\int_0^1 f'' \big(u(r)+\lambda (u(t_{i+1})-u(r)) \big)
 \big(u(t_{i+1})-u(r)\big)^2 (1-\lambda) {\rm d}\lambda.
\end{align*}
We shall estimate $J^{m+1}_{21i}$, $i=1,2,3,4$, successively.

Since for $\delta\in (3/2, 2)$, $\dot H^\delta\hookrightarrow \CC$, it follows by the dual argument that
\begin{align} \label{l1} 
\|x\|_{-\delta}\le C \|x\|_{L_x^1},
\quad x\in L_x^1.
\end{align}
This inequality, in conjunction with Minkovskii and Cauchy--Schwarz inequalities, the condition \eqref{con-f'}, the embedding \eqref{emb}, and the estimations \eqref{ana}, \eqref{reg-h1}, \eqref{reg-h1+}, and \eqref{reg-h2}, yields that 
\begin{align} \label{j211} 
J^{m+1}_{211} 
&\le \sum_{i=0}^m \int_{t_i}^{t_{i+1}} 
\|(-A)^\frac\delta2 S(t_{m+1}-r) \|_{\LL(H)}
\|DF(u(r)) (S(t_{i+1}-r)-{\rm Id}) u(r)\|_{L_\omega^p \dot H^{-\delta}} 
{\rm d}r \nonumber \\
&\le C \sum_{i=0}^m \int_{t_i}^{t_{i+1}} (t_{m+1}-r)^{-\frac\delta2}
\|DF(u(r)) (S(t_{i+1}-r)-{\rm Id}) u(r)\|_{L_\omega^p L_x^1}  {\rm d}r \nonumber \\
&\le C \sum_{i=0}^m \int_{t_i}^{t_{i+1}} (t_{m+1}-r)^{-\frac\delta2}
\|f'(u(r))\|_{L_\omega^{2p} L_x^2} 
\|(S(t_{i+1}-r)-{\rm Id}) u(r)\|_{L_\omega^{2p} L_x^2}  {\rm d}r \nonumber \\
&\le C \tau^\frac{1+\gamma}2 \|f'(u)\|_{L_t^\infty L_\omega^{2p} L_x^2} 
\|u\|_{L_t^\infty L_\omega^{2p} \dot H^{1+\gamma}} 
\Big(\int_0^{t_{m+1}} r^{-\frac\delta2}  {\rm d}r \Big) \nonumber \\
& \le C \tau^\frac{1+\gamma}2 \times
\begin{cases}
(1+\|u_0\|^{q-2}_1)(1+\|u_0\|_1), & \quad \gamma=0; \\
(1+\|u_0\|^{q-2}_1)(1+\|u_0\|^{q-1}_{1+\gamma}), & \quad \gamma\in (0,1); \\
(1+\|u_0\|^{q-2}_1)(1+\|u_0\|^{\rho (q-1)}_2+\|u_0\|^{(q-1)^2}_2), & \quad \gamma=1.
\end{cases}
\end{align}
Similar argument implies that 
\begin{align} \label{j212} 
J^{m+1}_{212} 
&\le C \sum_{i=0}^m \int_{t_i}^{t_{i+1}} \int_r^{t_{i+1}} 
(t_{m+1}-r)^{-\frac\delta2}
\|f'(u(r))\|_{L_\omega^{2p} L_x^2} 
\|f(u(\sigma))\|_{L_\omega^{2p} L_x^2} {\rm d}\sigma {\rm d}r \nonumber \\
&\le C \tau \|f'(u)\|_{L_t^\infty L_\omega^{2p} L_x^2} 
\|f(u)\|_{L_t^\infty L_\omega^{2p} L_x^2} 
\Big(\int_0^{t_{m+1}} r^{-\frac\delta2} {\rm d}r \Big)  \nonumber \\
& \le C \tau \big(1+\|u_0\|^{q-2}_1 \big) \big(1+\|u_0\|^{q-1}_1 \big).
\end{align}

To estimate the third term $J^{m+1}_{213}$, we apply stochastic Fubini theorem and the discrete and continuous Burkholder--Davis--Gundy inequalities \eqref{bdg-dis} and \eqref{bdg2}, respectively, to derive 
\begin{align*} 
J^{m+1}_{213} 
&=\Big\|\sum_{i=0}^m \int_{t_i}^{t_{i+1}} \Big[\int_{t_i}^{t_{i+1}} 
\chi_{[r,t_{i+1})}(\sigma) S(t_{m+1}-r) 
DF(u(r)) S(t_{i+1}-\sigma) G {\rm d}r\Big] {\rm d}W(\sigma) \Big\|_{L_\omega^p L_x^2} \\ 
&\le C \Big(\sum_{i=0}^m \int_{t_i}^{t_{i+1}} 
\Big\|\int_{t_i}^{t_{i+1}} S(t_{m+1}-r) 
DF(u(r)) S(t_{i+1}-\sigma) G {\rm d}r\Big\|^2_{L_\omega^p \LL_2^0} 
{\rm d}\sigma \Big)^\frac12. 
\end{align*}
Then by Cauchy--Schwarz inequality, the condition \eqref{con-f'}, and the estimation \eqref{reg-h1}, we obtain 
\begin{align} \label{j213} 
J^{m+1}_{213} 
&\le C \tau^\frac12 \Big(\sum_{i=0}^m \int_{t_i}^{t_{i+1}} 
\int_{t_i}^{t_{i+1}} \|S(t_{m+1}-r) DF(u(r)) S(t_{i+1}-\sigma) G\|_{L_\omega^p \LL_2^0}^2 {\rm d}r  
{\rm d}\sigma \Big)^\frac12  \nonumber \\
&\le C \tau^\frac12 \Big[\sum_{i=0}^m 
\Big(\int_{t_i}^{t_{i+1}} \|f'(u(r))\|^2_{L_\omega^p \dot H^{-1}}  {\rm d}r \Big)
\Big(\int_{t_i}^{t_{i+1}}  \|S(\sigma) G\|_{\LL_2^1}^2  {\rm d}\sigma \Big) \Big]^\frac12  \nonumber \\
&\le C \tau \|f'(u)\|_{L_t^\infty L_\omega^{2p} \dot H^{-1}} \|G\|_{\LL_2^1} 
\le C \big(1+\|u_0\|^{q-2}_1\big) \tau.
\end{align}
Finally, in terms of the estimations \eqref{ana} and \eqref{l1} and Minkovskii and H\"older inequalities, we estimate the last term $J^{m+1}_{214}$ similarly to $J^{m+1}_{211}$ by
\begin{align*} 
J^{m+1}_{214} 
&\le C \sum_{i=0}^m \int_{t_i}^{t_{i+1}} (t_{m+1}-r)^{-\frac\delta2}
\|R_F(u(r), u(t_{i+1}))\|_{L_\omega^p L_x^1} {\rm d}r \nonumber \\
&\le C \sum_{i=0}^m \int_{t_i}^{t_{i+1}} (t_{m+1}-r)^{-\frac\delta2}
\times \int_0^1 \big\|f''\big(u(r)+\lambda (u(t_{i+1})-u(r)) \big) \big\|_{L_\omega^{3p} L_x^3} \nonumber \\
&\qquad \qquad \qquad \big\|u(t_{i+1})-u(r) \big\|^2_{L_\omega^{3p} L_x^3} (1-\lambda) {\rm d}\lambda {\rm d}r.
\end{align*}
It follows from the condition \eqref{con-f''}, the embeddings \eqref{emb} and $\dot H^{1/2} \hookrightarrow L_x^3$, the moments' estimation \eqref{reg-h1}, and the H\"older estimation \eqref{hol-u1} that 
\begin{align} \label{j214} 
J^{m+1}_{214} 
&\le C \tau^{(\frac12+\gamma) \wedge 1} \Big(1+\|u\|^{\widetilde{q}-3}_{L_t^\infty L_\omega^{3p(\widetilde{q}-3)} L_x^{3(\widetilde{q}-3)}} \Big)
\Big(\int_0^{t_{m+1}} r^{-\frac\delta2} {\rm d}r\Big)  \nonumber \\
& \quad \times \Big[\sup_{0\le s < t \le T} \frac{\|u(t)-u(s) \|^2_{L_\omega^{3p} \dot H^{1/2}}} 
{|t-s|^{(1/2+\gamma) \wedge 1}} \Big] \nonumber \\
& \le C \tau^{(\frac12+\gamma) \wedge 1} \times
\begin{cases}
\big(1+\|u_0\|^{\widetilde{q}+q-4}_1 \big)  
\big(1+\|u_0\|^{q-1}_{1+\gamma} \big), & \quad \gamma\in [0,1); \\
\big(1+\|u_0\|^{\widetilde{q}+q-4}_1 \big) 
\big(1+\|u_0\|^{\rho (q-1)}_2+\|u_0\|^{(q-1)^2}_2\big), & \quad \gamma=1.
\end{cases}
\end{align}
Collecting the above four estimations \eqref{j211}-\eqref{j214} together, we obtain \eqref{j21} and complete the proof.
\end{proof}

\subsection{Strong Convergence Rate of DIEMG Scheme}
\label{sec5.2}

Next, we consider the DIEMG scheme \eqref{milstein} for Eq. \eqref{see} with general multiplicative noise under an infinite-dimensional analog of the commutativity type condition in the SODE setting.
Such commutativity condition was first introduced in \cite[Assumption 3]{JR15(FCM)} where the authors observed in \cite[Section 4]{JR15(FCM)} that a certain class of semilinear SPDEs with multiplicative trace class noise naturally fulfills such commutativity condition. 
The following assumption holds true automatically in the case of the additive noise case $G(u)\equiv G\in \LL_2^0$ for any $u\in H$.

\begin{ap} \label{ap-g+}
$G: H \rightarrow \LL_2^0$ be a twice continuously Fr\'echet differentiable mapping such that for every $z\in \dot H^1$, $DG(z)\in \LL(H; \LL_2^0)$ and $D^2G(z)\in \LL^{\otimes 2}(\dot H^\beta; \LL_2^0)$ for some $\beta\in [0, \frac{1+\gamma}2]$.
Moreover, there exists a constant $L'_g\in \rr_+$ such that 
\begin{align} 
\|DG(z)\|_{\LL(L_x^2; \LL_2^0)}
+\|D^2G(z)\|_{\LL^{\otimes 2}(\dot H^\beta; \LL_2^0)} \le L_g', 
& \quad z\in \dot H^1, \label{con-g12}  \\
\|DG(u) G(u)-DG(v) G(v)\|_{HS(U_0; \LL_2^0)}
\le L_g' \|u-v\|,
& \quad u,v\in H. \label{con-g-com}  
\end{align}
\end{ap}

\begin{ex} \label{ex-g}
Assume that $g:\rr\rightarrow \rr$ a twice continuously differentiable function such that   
\begin{align} \label{ex-g+}
\sup_{x\in \rr} |g'(x)|+\sup_{x\in \rr} |g''(x)|
+\sup_{x,y\in \rr, x\neq y} \frac{\|g'(x) g(x)-g'(y) g(y)\|}{|x-y|}<\infty,
\end{align}
$g_k\in W^{1,\infty}$ for each $k\in \nn_+$ such that  
\begin{align} \label{ex-gq}
\sum_{k\in \nn_+} \lambda^{\bf Q}_k \|g_k\|^2_{W^{1,\infty}}<\infty,
\end{align}
and that either $g(0)=0$ or $g_k$ vanishes on the boundary $\partial \OOO$, then Assumptions \ref{ap-g} and \ref{ap-g+} hold true.
In fact, for any $u,v\in L_x^2$ and $z\in \dot H^1$, we have
\begin{align*} 
\|G(u)-G(v)\|_{\LL_2^0} 
&=\Big(\sum_{k\in \nn_+} \lambda^{\bf Q}_k \|(g(u)-g(v)) g_k\|^2 \Big)^\frac12 \\
&\le \sup_{x\in \rr} |g'(x)| \Big(\sum_{k\in \nn_+} \lambda^{\bf Q}_k \|g_k\|_{L_x^\infty}^2\Big)^\frac12  \|u-v\|,
\end{align*}
and
\begin{align*}
\|G(z) \|^2_{\LL_2^1} 
&\le C \sum_{k\in \nn_+} \lambda^{\bf Q}_k 
\Big(\|g(z) g_k\|^2+\|g'(z) \nabla z g_k+g(z) g'_k\|^2 \Big) \\
&\le C \Big(1+\sup_{x\in \rr} |g'(x)|^2 \Big) \sum_{k\in \nn_+} \lambda^{\bf Q}_k 
\|g_k\|_{W_x^{1,\infty}}^2 (1+\|z\|_1)^2,
\end{align*}
which show Assumption \ref{ap-g}.
On the other hand, for any $z\in \dot H^1$,
\begin{align*}
&\|DG(z)\|_{\LL(L_x^2; \LL_2^0)}
+\|D^2G(z)\|_{\LL^{\otimes 2}(\dot H^\beta; \LL_2^0)} \\
&=\sup_{v\in L_x^2} \frac{\Big(\sum_{k\in \nn_+} \lambda^{\bf Q}_k
\|g'(z) v g_k\|^2 \Big)^\frac12}{\|v\|}
+\sup_{u,v\in \dot H^\beta} \frac{\Big(\sum_{k\in \nn_+} \lambda^{\bf Q}_k
\|(g''(z) u v g_k\|^2 \Big)^\frac12}{\|u\|_\beta \|v\|_\beta} \\
&\le C \Big(\sum_{k\in \nn_+} \lambda^{\bf Q}_k
\|g_k\|_{L_x^\infty}^2 \Big)^\frac12 
\Big[\sup_{x\in \rr} |g'(x)|+\sup_{x\in \rr} |g''(x)| \cdot
\sup_{u,v\in \dot H^\beta} \frac{\|u\|_{L_x^4} \|v\|_{L_x^4}}{\|v\|_\beta \|z\|_\beta} \Big].
\end{align*}
Due to the assumptions \eqref{ex-g+}-\eqref{ex-gq} and the embedding $\dot H^\beta\hookrightarrow L_x^4$ with $\beta=1/2$ when $d=1,2$ and 
$\beta=3/4$ when $d=3$, we have
\begin{align*}
&\|DG(z)\|_{\LL(L_x^2; \LL_2^0)}
+\|D^2G(z)\|_{\LL^{\otimes 2}(\dot H^\beta; \LL_2^0)} \\
&\le C  \sup_{v\in L_x^2} \Big(\sum_{k\in \nn_+} \lambda^{\bf Q}_k
\|g_k\|_{L_x^\infty}^2 \Big)^\frac12 \Big[\sup_{x\in \rr} |g'(x)|+\sup_{x\in \rr} |g''(x)| \Big].
\end{align*}
Moreover, for any $u,v\in L_x^2$,
\begin{align*}
&\|DG(u) G(u)-DG(v) G(v)\|_{HS(U_0; \LL_2^0)} \\
&=\Big(\sum_{k\in \nn_+} \sum_{l\in \nn_+} 
\lambda^{\bf Q}_k \lambda^{\bf Q}_l
\|DG(u) (G(u) g_k) g_l-DG(v) (G(v) g_l) g_k\|^2 \Big)^\frac12 \\
&\le \Big(\sum_{k\in \nn_+}  \lambda^{\bf Q}_k \|g_k\|_{L_x^\infty}^2 \Big) 
\|g'(u) g(u)-g'(v) g(v)\|
\le \Big(\sum_{k\in \nn_+}  \lambda^{\bf Q}_k \|g_k\|_{L_x^\infty}^2 \Big) 
\|u-v\|.
\end{align*}
The above two estimations show Assumption \ref{ap-g+} with $\gamma\in [0,1]$ and $\beta\in [0,1]$ when $d=1,2$ and $\gamma\in [1/2,1]$ and 
$\beta\in [0,(1+\gamma)/2]$ when $d=3$.

\end{ex}

We introduce a version $\widetilde{U}_h^{m+1}$ of the auxiliary process 
$\widetilde{u}_h^{m+1}$ given by \eqref{aux1} in consistent with the DIEMG scheme \eqref{full+}, which is defined as 
\begin{align}\label{aux2}
\widetilde{U}_h^{m+1}
&=S_{h,\tau}^{m+1} u_h^0
+\tau \sum_{i=0}^m S_{h,\tau}^{m+1-i} \PP_h F(u(t_{i+1}))
+\sum_{i=0}^m S_{h,\tau}^{m+1-i} \PP_h G(u(t_i)) \delta_i W
\nonumber \\
&\quad +\sum_{i=0}^m S_{h,\tau}^{m+1-i} 
\PP_h DG(u(t_i)) G(u(t_i)) 
\Big[\int_{t_i}^{t_{i+1}} (W(r)-W(t_i)) {\rm d}W(r) \Big],
\end{align}
for $m\in \zz_{M-1}$, where the terms $U_h^{i+1}$ and $U_h^i$ in the discrete deterministic and stochastic convolutions of \eqref{milstein-sum} are replaced by $u(t_{i+1})$ and $u(t_i)$, respectively.
The boundedness \eqref{reg-h1} and \eqref{shm} imply the solvability and uniform boundedness of $\widetilde{U}_h^m$:
\begin{align} \label{wide-Uh}
\sup_{m\in \zz_M} \ee\Big[ \|\widetilde{U}_h^m\|^p_1 \Big]
\le C \Big(1+\|u_0\|_1^{p(q-1)} \Big),
\end{align}
provided $u_0\in \dot H^1$ and Assumptions \ref{ap-f}-\ref{ap-g} and \ref{ap-g+} hold.
We begin with the strong error estimation between $u$ and $\{\widetilde{U}_h^m\}_{m\in \zz_M}$.

\begin{lm} \label{lm-u-Uhm}
Let $\gamma\in [0,1)$, $u_0\in \dot H^{1+\gamma}$, and Assumptions \ref{ap-f}-\ref{ap-g} and \ref{ap-f+}-\ref{ap-g+} hold. 
Then for any $p\ge 1$, there exists a constant $C$ such that
\begin{align} \label{u-Uhathm}
\sup_{m\in \zz_M} \|u(t_m)-\widetilde{U}_h^m\|_{L_\omega^p L_x^2}
\le C(h^{1+\gamma}+\tau^\frac{1+\gamma}2).
\end{align}
Assume furthermore that $u_0\in \dot H^2$ and Assumption \ref{ap-g2} holds, then \eqref{u-Uhathm} is valid with $\gamma=1$.
\end{lm}

\begin{proof}
For $m\in \zz_{M-1}$, define $\widetilde{J}^{m+1}:=\|u(t_{m+1})-\widetilde{U}_h^{m+1}\|_{L_\omega^p L_x^2}$.
By the proofs of Lemma \ref{lm-u-uhm} and Corollary \ref{cor-u-uhm}, we have
\begin{align*} 
\widetilde{J}^{m+1}\le \sum_{i=1}^2 J^{m+1}_i+\widetilde{J}^{m+1}_3,
\end{align*}
where $J^{m+1}_i$, $i=1,2$, are defined from \eqref{j} and satisfy the estimate
\begin{align} \label{j1+2} 
J^{m+1}_1+J^{m+1}_2 
&\le C \big(h^{1+\gamma}+\tau^\frac{1+\gamma}2 \big),
\end{align}
by \eqref{j1}, \eqref{j22}, \eqref{j22+}, and \eqref{j21}, and the last term is defined as
\begin{align*} 
\widetilde{J}^{m+1}_3 
&:=\Big\|\int_0^{t_{m+1}} S(t_{m+1}-r) G(u(r)) {\rm d}W(r) 
-\sum_{i=0}^m S_{h,\tau}^{m+1-i} \PP_h G(u(t_i)) \delta_i W 
\nonumber \\
&\quad -\sum_{i=0}^m S_{h,\tau}^{m+1-i} 
\PP_h DG(u(t_i)) G(u(t_i)) 
\Big[\int_{t_i}^{t_{i+1}} (W(r)-W(t_i)) {\rm d}W(r) \Big]\Big\|_{L_\omega^p L_x^2}.
\end{align*}
Thus to prove \eqref{u-Uhathm}, it suffices to show that 
\begin{align}  \label{j3hat}
\widetilde{J}^{m+1}_3 
&\le C \big(h^{1+\gamma}+\tau^\frac{1+\gamma}2 \big).
\end{align}

The term $\widetilde{J}^{m+1}_3$ can be decomposed as 
\begin{align*} 
\widetilde{J}^{m+1}_3 
&\le \Big\|\sum_{i=0}^m \int_{t_i}^{t_{i+1}} 
H^{m+1}_i(r) {\rm d}W(r) \Big\|_{L_\omega^p L_x^2} \\
&\quad +\Big\|\int_0^{t_{m+1}} E_{h,\tau}(t_{m+1}-r) G(u(r)) {\rm d}W(r) \Big\|_{L_\omega^p L_x^2} 
:=\widetilde{J}^{m+1}_{31}+J^{m+1}_{32},
\end{align*}
where
\begin{align*} 
H^{m+1}_i(r)
=S_{h,\tau}^{m+1-i} \PP_h 
[G(u(r))-G(u(t_i))-DG(u(t_i)) G(u(t_i)) (W(r)-W(t_i))].
\end{align*}
By \eqref{j3} and \eqref{j32}, the second term $J^{m+1}_{32}$ has the estimation
\begin{align} \label{j32+} 
J^{m+1}_{32} 
\le C \Big(h^{1+\gamma}+\tau^\frac{1+\gamma}2 \Big) \times 
\begin{cases}
\big(1+\|u_0\|_1^{q-1} \big),
&\quad \gamma\in [0,1), \\
\big(1+\|u_0\|^{q-1}_2 \big),
&\quad \gamma=1.
\end{cases}
\end{align}
It remains to show a similar estimation for 
$\widetilde{J}^{m+1}_{31}$.
It is clear that 
\begin{align*} 
H^{m+1}_i(r)
&=S_{h,\tau}^{m+1-i} \PP_h DG(u(t_i)) [u(r)-u(t_i)-G(u(t_i) (W(r)-W(t_i))]  \\
&\quad +S_{h,\tau}^{m+1-i} \PP_h \int_0^1 D^2 G\big(u(t_i)+\lambda (u(r)-u(t_i)) \big) \\
&\qquad \qquad \big(u(r)-u(t_i), u(r)-u(t_i) \big) (1-\lambda) {\rm d}\lambda,
\quad r\in [t_1,t_{i+1}).
\end{align*}
As a consequence of the discrete and continuous Burkholder--Davis--Gundy inequalities \eqref{bdg-dis} and \eqref{bdg2}, respectively, we have
\begin{align} 
\widetilde{J}^{m+1}_{31} 
&\le \sum_{i=1}^2 \sqrt{\widetilde{J}^{m+1}_{31i}}, 
\end{align}
where
\begin{align*} 
\widetilde{J}^{m+1}_{311}
&:=\sum_{i=0}^m \int_{t_i}^{t_{i+1}} 
\Big\|S_{h,\tau}^{m+1-i} \PP_h DG(u(t_i)) 
\Big[u(r)-u(t_i)-\int_{t_i}^r G(u(t_i)) {\rm d}W(s) \Big] \Big\|^2_{L_\omega^p \LL_2^0} {\rm d}r, \\
\widetilde{J}^{m+1}_{312}
&:=\sum_{i=0}^m \int_{t_i}^{t_{i+1}} 
\Big\|S_{h,\tau}^{m+1-i} \PP_h \int_0^1 D^2 G\big(u(t_i)+\lambda (u(r)-u(t_i)) \big) \\
&\qquad \qquad  \qquad \big(u(r)-u(t_i), u(r)-u(t_i) \big) (1-\lambda) {\rm d}\lambda  \Big\|^2_{L_\omega^p \LL_2^0} {\rm d}r.
\end{align*}

We begin with the first term $\widetilde{J}^{m+1}_{311}$.
By the boundedness \eqref{ana}, the condition \eqref{con-g12}, the mild formulation \eqref{mild}, and H\"older inequality, we have
\begin{align*}
\widetilde{J}^{m+1}_{311}
&\le C \sum_{i=0}^m \int_{t_i}^{t_{i+1}} 
\|u(r)-u(t_i)-G(u(t_i) (W(r)-W(t_i)) \|^2_{L_\omega^p L_x^2} {\rm d}r \\
&\le C \sum_{i=0}^m \int_{t_i}^{t_{i+1}} 
\|(S(r-t_i)-{\rm Id}) u(t_i)\|^2_{L_\omega^p L_x^2} {\rm d}r  \\
&\quad +C \sum_{i=0}^m \int_{t_i}^{t_{i+1}} \Big\|\int_{t_i}^r S(r-t_i) F(u(s) {\rm d}s \Big\|^2_{L_\omega^p L_x^2} {\rm d}r  \\
&\quad +C \sum_{i=0}^m \int_{t_i}^{t_{i+1}} \Big\|\int_{t_i}^r (S(r-s)-{\rm Id}) G(u(s)) {\rm d}W(s) \Big\|^2_{L_\omega^p L_x^2} {\rm d}r    \\
&\quad +C \sum_{i=0}^m \int_{t_i}^{t_{i+1}} \Big\|\int_{t_i}^r 
[G(u(s))-G(u(t_i))] {\rm d}W(s) \Big\|^2_{L_\omega^p L_x^2}  {\rm d}r.
\end{align*}
The conditions \eqref{con-f1} and \eqref{con-g}, the embedding \eqref{emb}, and the estimations \eqref{ana}, \eqref{reg-h1}, \eqref{hol-u}, \eqref{reg-h1+}, and \eqref{reg-h2} imply that 
\begin{align} \label{j311}
\widetilde{J}^{m+1}_{311}    
&\le C \tau^{1+\gamma} \|u\|^2_{L_t^\infty L_\omega^p \dot H^{1+\gamma}}
+C \tau^2 \Big(1+\|u\|^{q-1}_{L_t^\infty L_\omega^{p(q-1)} \dot H^1}\Big)^2 \nonumber  \\
&\quad +C \tau^2 \Big(1+\|u\|_{L_t^\infty L_\omega^p \LL_2^1} \Big)^2
+C \tau^2 \sup_{0\le s < t \le T} \frac{\|u(t)-u(s)\|^2_{L_\omega^p L_x^2}}{(t-s)}
\nonumber \\
&\le C \tau^{1+\gamma} \times 
\begin{cases}
\big(1+\|u_0\|_1^{q-1} \big)^2,
&\quad \gamma\in [0,1), \\
\big(1+\|u_0\|^{\rho (q-1)}_2+\|u_0\|^{(q-1)^2}_2 \big)^2,
&\quad \gamma=1.
\end{cases}
\end{align}
Next, we estimate the second term $\widetilde{J}^{m+1}_{312}$ by the boundedness \eqref{shm}, \eqref{con-g12} with $\beta\in [0, \frac{1+\gamma}2)$, Cauchy--Schwarz inequality, and the H\"older estimation \eqref{hol-u}:
\begin{align}  \label{j313} 
\widetilde{J}^{m+1}_{312}
&\le C \sup_{z\in \dot H^1} \|D^2G(z)\|^2_{\LL^{\otimes 2}(\dot H^\beta; \LL_2^0)}
\Big[\sum_{i=0}^m \int_{t_i}^{t_{i+1}} 
\big\|u(r)-u(t_i) \big\|^4_{L_\omega^{2p} \dot H^\beta} {\rm d}r \Big]  \nonumber \\
& \le C \sup_{0\le s < t \le T} 
\frac{\|u(t)-u(s)\|^4_{L_\omega^{2p} L_x^2}} {|t-s|^{2(1+\gamma-\beta)\wedge 2}}
\Big[\sum_{i=0}^m \int_{t_i}^{t_{i+1}} (r-t_i)^{2(1+\gamma-\beta)\wedge 2} {\rm d}r \Big]  \nonumber \\
& \le C \tau^{2(1+\gamma-\beta)\wedge 2} (1+\|u_0\|^{q-1}_1)^4
\le C \tau^{1+\gamma} (1+\|u_0\|^{q-1}_1 )^4.
\end{align}

Collecting the estimations \eqref{j311}-\eqref{j313}, we derive 
\begin{align} \label{j31+} 
\widetilde{J}^{m+1}_{31} \le C \tau^\frac{1+\gamma}2.
\end{align}
The above estimation \eqref{j31+}, in combination with \eqref{j32+}, shows \eqref{j3hat} and thus completes the proof.
\end{proof}

\begin{rk} 
Under the conditions of Lemma \ref{lm-u-Uhm} with $\gamma\in [0,1]$ and 
$\beta\in [0,1]$, the proof of Lemma \ref{lm-u-Uhm} yields that, for any $p\ge 1$, there exists a constant $C$ such that
\begin{align*}
\sup_{m\in \zz_M} \|u(t_m)-\widetilde{U}_h^m\|_{L_\omega^p L_x^2}
\le C (h^{1+\gamma}+\tau^{\frac{1+\gamma}2\wedge (1+\gamma-\beta)} ).
\end{align*}
\end{rk}

Now we can prove Theorem \ref{main2} by combining Lemma \ref{lm-u-Uhm} with a variational approach as in the proof of Theorem \ref{main1}.

\begin{proof} [Proof of Theorem \ref{main2}]
Let $m\in \zz_{M-1}$.
In terms of Minkovskii inequality and the estimation \eqref{u-Uhathm}, it suffices to prove that 
\begin{align}  \label{u-Uhm2}
\sup_{m\in \zz_M} \|\widetilde{U}_h^m-U_h^m\|_{L_\omega^p L_x^2}
\le C \big(h^{1+\gamma}+\tau^\frac{1+\gamma}2 \big),
\quad \gamma\in [0,1].
\end{align}

Define $\widehat{e}_h^{m+1}:=\widetilde{U}_h^{m+1}-U_h^{m+1}$.
Then $\widehat{e}_h^{m+1}\in V_h$
with vanishing initial datum $\widehat{e}_h^0=0$.
In terms of Eq. \eqref{milstein} and \eqref{aux2}, it is clear that 
\begin{align*}
\widetilde{U}_h^{m+1}
&=\widetilde{U}_h^m
+\tau A_h \widetilde{U}_h^{m+1}
+\tau \PP_h F(u(t_{m+1}))
+\PP_h G(u(t_m)) \delta_m W
\nonumber \\
&\quad +\PP_h DG(u(t_m)) G(u(t_m)) 
\Big[\int_{t_m}^{t_{m+1}} (W(r)-W(t_m)) {\rm d}W(r) \Big], \\
U_h^{m+1}
&=U_h^m+\tau A_h U_h^{m+1}
+\tau \PP_h F(U_h^{m+1})
+\PP_h G(U_h^m) \delta_m W \nonumber \\
&\quad +\PP_h DG(U_h^m) G(U_h^m) 
\Big[\int_{t_m}^{t_{m+1}} (W(r)-W(t_m)) {\rm d}W(r) \Big].
\end{align*}
Consequently, 
\begin{align*}
 \widehat{e}_h^{m+1}-\widehat{e}_h^m
&=\tau A_h \widehat{e}_h^{m+1}
+\tau \PP_h (F(u(t_{m+1}))-F(U_h^{m+1})) \nonumber \\
&\quad +\PP_h (G(u(t_m))-G(U_h^m))\delta_m W
+\int_{t_m}^{t_{m+1}} \PP_h H^{m+1} {\rm d}W(r),
\end{align*}
where 
\begin{align*}
H^{m+1}(r):=\int_{t_m}^r \Big[DG(u(t_m)) G(u(t_m))-DG(U_h^m) G(U_h^m)\Big] {\rm d}W(s),
\quad r\in [t_m, t_{m+1}).
\end{align*}
Multiplying $\widehat{e}_h^{m+1}$ on both sides of the above equation as in \eqref{eq-u-uhm}, we obtain
\begin{align} \label{eq-u-Uhm}
& \<\widehat{e}_h^{m+1}-\widehat{e}_h^m, \widehat{e}_h^{m+1}\>
+\tau \|\nabla \widehat{e}_h^{m+1}\|^2
=\tau ~ {_{-1}}\<F(u(t_{m+1}))-F(U_h^{m+1}), \widehat{e}_h^{m+1}\>_1 \nonumber \\
&\quad +\< (G(u(t_m))-G(U_h^m))\delta_m W, \widehat{e}_h^{m+1}\>
+\int_{t_m}^{t_{m+1}} \<H^{m+1} {\rm d}W(r), \widehat{e}_h^{m+1}\>.
\end{align}

Similarly to \eqref{k0}, we have
\begin{align}\label{k0+}
\ee\Big[\<\widehat{e}_h^{m+1}-\widehat{e}_h^m, \widehat{e}_h^{m+1}\> \Big]
&= \frac12 \Big(\ee\Big[\|\widehat{e}_h^{m+1}\|^2\Big]
-\ee\Big[\|\widehat{e}_h^m\|^2\Big]\Big)
+\frac12 \ee\Big[\|\widehat{e}_h^{m+1}-\widehat{e}_h^m\|^2\Big].
\end{align}
By analogous arguments in the proof of Theorem \ref{main1}, we have 
\begin{align} \label{k1+}
& \tau ~ {_{-1}}\<F(u(t_{m+1}))-F(U_h^{m+1}), \widehat{e}_h^{m+1}\>_1
\le \frac12\tau \|\nabla \widehat{e}_h^{m+1}\|^2 
+L_f \tau \|\widehat{e}_h^{m+1}\|^2 \nonumber \\
&\quad +C \tau \|u(t_{m+1})-\widetilde{U}_h^{m+1}\|^2 
\big(1+\|u(t_{m+1})\|^{q-2}_1 +\|\widetilde{U}_h^{m+1}\|^{q-2}_1 \big),
\end{align}
and 
\begin{align} \label{k2+}
& \< (G(u(t_m))-G(U_h^m))\delta_m W, \widehat{e}_h^{m+1}\> \nonumber \\
&\le \frac14  \ee\Big[\|\widetilde{e}_h^{m+1}-\widetilde{e}_h^m\|^2\Big]
+2L_g^2 \ee\Big[\|u(t_m)-\widetilde{U}_h^m\|^2 \Big] \tau
+2L_g^2 \ee\Big[\|\widetilde{e}_h^m\|^2 \Big] \tau.
\end{align}

To deal with the last stochastic term in \eqref{eq-u-Uhm}, we use the martingale property of the stochastic integral, Cauchy--Schwarz inequality, and It\^o isometry as in \eqref{j311}:
\begin{align*}
\ee\Big[ \int_{t_m}^{t_{m+1}} \< H^{m+1} {\rm d}W(r), \widehat{e}_h^{m+1}\> \Big]  
\le \frac14 \ee\Big[\|\widetilde{e}_h^{m+1}-\widetilde{e}_h^m\|^2\Big]
+C \int_{t_m}^{t_{m+1}} 
\ee\Big[\|H^{m+1}\|^2_{\LL_2^0} \Big] {\rm d}r.
\end{align*}
By It\^o isometry, the condition \eqref{con-g-com}, and Cauchy--Schwarz inequality, we get 
\begin{align*}
& \int_{t_m}^{t_{m+1}} 
\ee\Big[\|H^{m+1} \|^2_{\LL_2^0} \Big] {\rm d}r \\
&= \int_{t_m}^{t_{m+1}} \ee\Big[ \Big\|\int_{t_m}^r \Big[DG(U_h^m) G(U_h^m)-DG(u(t_m)) G(u(t_m)) \Big] {\rm d}W(s)\Big \|^2_{\LL_2^0} \Big] {\rm d}r \\ 
&= \frac12 \tau^2 \ee\Big[ \|DG(U_h^m) G(U_h^m)-DG(u(t_m)) G(u(t_m))\|^2_{HS(U_0; \LL_2^0)} \Big] \\
&= \frac{(L'_g)^2}2 \tau^2 \ee\Big[ \|u(t_m)-U_h^m\|^2  \Big]
\le L_g'^2 \tau^2 \ee\Big[ \|u(t_m)-\widetilde{U}_h^m\|^2 \Big]
+L_g'^2 \tau^2 \ee\Big[ \|\widetilde{e}_h^m\|^2 \Big],
\end{align*}
and thus 
\begin{align} \label{k3+}
& \ee\Big[ \int_{t_m}^{t_{m+1}} \< H^{m+1} {\rm d}W(r), \widehat{e}_h^{m+1}\> \Big] \nonumber  \\
& \le \frac14 \ee\Big[\|\widetilde{e}_h^{m+1}-\widetilde{e}_h^m\|^2\Big]
+L_g'^2 \tau^2 \ee\Big[ \|u(t_m)-\widetilde{U}_h^m\|^2 \Big]
+L_g'^2 \tau^2 \ee\Big[ \|\widetilde{e}_h^m\|^2 \Big].
\end{align}
Now taking expectation in \eqref{eq-u-Uhm}, we use \eqref{k0+}-\eqref{k3+} and H\"older inequality to get 
\begin{align*} 
& \frac12 \Big(\ee\Big[\|\widehat{e}_h^{m+1}\|^2\Big]
-\ee\Big[\|\widehat{e}_h^m\|^2\Big]\Big) \nonumber \\
&\le L_f \ee\Big[\|\widehat{e}_h^{m+1}\|^2\Big] \tau
+C \Big(\ee\Big[\|u(t_{m+1})-\widetilde{U}_h^{m+1}\|^4\Big]\Big)^\frac12 \tau
+L_g'^2 \tau^2 \ee\Big[ \|u(t_m)-\widetilde{U}_h^m\|^2 \Big] \nonumber \\
&\quad \times \Big(1+\Big(\ee\Big[\|u(t_{m+1})\|_1^{2(q-2)}\Big] \Big)^\frac12
+\ee\Big[\|\widetilde{U}_h^{m+1}\|_1^{2(q-2)}\Big] \Big)^\frac12 \Big)  \nonumber \\
&\quad +2L_g^2 \ee\Big[\|u(t_m)-\widetilde{U}_h^m\|^2 \Big] \tau
+\Big( 2L_g^2 \tau+L_g'^2 \tau^2 \Big) \ee\Big[\|\widehat{e}_h^m\|^2 \Big] .
\end{align*}
Then by the estimations \eqref{reg-h1}, \eqref{wide-Uh}, and \eqref{u-Uhathm}, we obtain
\begin{align*}
\big(1-2L_f \tau \big) \ee\Big[\|\widehat{e}_h^{m+1}\|^2\Big]
\le C \tau \big(h^{1+\gamma}+\tau^\frac{1+\gamma}2\big)^2
+\big(1+C \tau \big) \ee\Big[\|\widehat{e}_h^m\|^2\Big].
\end{align*}
We conclude \eqref{u-Uhm2} by the same arguments as in the proof of Theorem \ref{main1}.
\end{proof}

\subsection{Milstein--Spectral Galerkin Scheme}
\label{sec5.3}

Similarly to Section \ref{sec4.3}, the parallel result in Theorem \ref{main2} for DIEMG scheme \eqref{milstein} is also valid for the spectral Galerkin-based Milstein scheme of Eq. \eqref{see}.

The DIE Milstein spectral Galerkin (DIEMSG) scheme of Eq. \eqref{see} is to find a $V_h$-valued discrete process $\{U_N^m:\ N\in \nn_+; m\in \zz_M\}$ such that 
 \begin{align}\label{milstein-n} \tag{DIEMSG}
\begin{split}
U_N^{m+1}
&=U_N^m+\tau A_N U_N^{m+1}
+\tau \PP_N F(U_N^{m+1})
+\PP_N G(U_N^m) \delta_m W \nonumber \\
&\quad +\PP_N DG(U_N^m) G(U_N^m) 
\Big[\int_{t_m}^{t_{m+1}} (W(r)-W(t_m)) {\rm d}W(r) \Big],
\end{split}
\end{align}
with initial datum $U_N^0=\PP_N u_0$. 
By the arguments in the proof of Theorem \ref{main2}, one can show the following strong convergence results for this scheme with sharp rates.

\begin{tm}   \label{u-UNm}
Let $\gamma\in [0,1)$, $u_0\in \dot H^{1+\gamma}$, and Assumptions \ref{ap-f}-\ref{ap-g} and \ref{ap-f+}-\ref{ap-g+} hold.
There exist a constant $C$ such that 
\begin{align} \label{u-Unm}
\sup_{m\in \zz_M} \|u(t_m)-U_N^m\|_{L_\omega^p L_x^2}
\le C(N^{-\frac{1+\gamma} d}+\tau^\frac{1+\gamma}2).
\end{align}
Assume furthermore that $u_0\in \dot H^2$ and Assumption \ref{ap-g2} holds, then \eqref{u-Unm} is valid with $\gamma=1$.
\end{tm}

\section*{Acknowledgements}

We thank the anonymous referee for very helpful remarks and suggestions.
The first author is partially supported by Hong Kong RGC General Research Fund, No. 16307319, and the UGC Research Infrastructure Grant, No. IRS20SC39. The second author is partially supported by Hong Kong RGC General Research Fund, No. 15325816, and the Hong Kong Polytechnic University Start-up Fund for New Recruits, No. 1-ZE33.

\bibliographystyle{amsplain}
\bibliography{bib}

\end{document}